\documentclass{amsart}
\usepackage{amsmath,amssymb,amsthm,mathtools}
\usepackage{graphicx,enumitem,microtype}
\usepackage{hyperref}
\newtheorem{theorem}{Theorem}[section]
\newtheorem{proposition}[theorem]{Proposition}

\newtheorem{corollary}[theorem]{Corollary}
\newtheorem{lemma}[theorem]{Lemma}

\theoremstyle{definition}

\newtheorem{example}[theorem]{Example}
\theoremstyle{remark}
\newtheorem{remark}[theorem]{Remark}
\numberwithin{equation}{section}
\newcommand{\mc}{\mathcal}
\newcommand{\mb}{\mathbb}
\newcommand{\wh}{\widehat}
\newcommand{\F}{\mb F}
\newcommand{\HFK}{\wh{HFK}}
\newcommand{\Z}{\mathbb Z}
\newcommand{\Q}{\mathbb Q}
\newcommand{\Kh}{\operatorname{Kh}}
\newcommand{\rKh}{\widetilde{\operatorname{Kh}}}
\newcommand{\rC}{\widetilde C}
\newcommand{\rank}{\operatorname{rank}}
\newcommand{\im}{\operatorname{im}}
\newcommand{\coker}{\operatorname{coker}}
\newcommand{\Tor}{\operatorname{Tor}}
\newcommand{\lk}{\operatorname{lk}}
\newcommand{\Tot}{\operatorname{Tot}}
\begin{document}
\title[Homology of fibered positive links]{Khovanov and knot Floer homology of fibered positive links}
\author{Zhechi Cheng}
\address{School of Mathematics and Statistics, Wuhan University, Wuhan, Hubei, China}
\email{\href{mailto:zcheng@whu.edu.cn}{zcheng@whu.edu.cn}}
\begin{abstract}
For fibered positive links, we compute the reduced and unreduced
Khovanov homology in homological grading two over both $\Z$ and $\F_2$,
and the knot Floer homology in the penultimate Alexander grading.
Our Khovanov calculations, combined with the result of
Przytycki--Sazdanovi\'c, prove that every nontrivial semi-adequate
knot has $2$-torsion in its unreduced Khovanov
homology. This establishes Shumakovitch's conjecture for
semi-adequate knots, which include all nontrivial positive knots.
The proofs involve induction using the skein exact sequences, with the Hopf
link as the base case.
\end{abstract}
\maketitle
\setcounter{tocdepth}{1}
\tableofcontents
\section{Introduction}\label{sec:introduction}

A knot or link is \emph{positive} if it admits an oriented diagram
with only positive crossings, and \emph{fibered} if its complement
is a surface bundle over a circle. The author and Hedden \cite{CH}
computed the next-to-top knot Floer homology of nonsplit positive
braid closures, which are fibered positive links. We extend this
computation to all fibered positive links.

In Khovanov homology, Shumakovitch conjectured that all nontrivial
knots have $2$-torsion in their unreduced Khovanov homology
\cite[Conjecture~1]{Shumakovitch14}.
Fibered positive knots supply the missing case of this conjecture
for positive and semi-adequate knots. Kegel, Manikandan,
Mousseau, and Silvero \cite[Theorem~1.1]{KMMS} proved the existence
of $2$-torsion for nonfibered positive knots. For semi-adequate
knots, the results of Przytycki--Sazdanovi\'c
\cite{PrzytyckiSazdanovic14} reduce the remaining problem to
the case of fibered positive knots.

In this paper, we compute the reduced and unreduced even Khovanov homology in
homological grading two as well as the Maslov-graded next-to-top knot Floer
homology of fibered positive links. Throughout the manuscript,
we take Khovanov homology over $R=\Z$ or $\F_2$ unless stated
otherwise. We always take knot Floer homology over $\F=\F_2$.
We use $\{d\}$ to raise quantum grading by $d$.
The notation $[r]$ raises homological grading by $r$ in Khovanov
homology and Maslov grading by $r$ in knot Floer homology.

\subsection{Khovanov homology}
Let $\Kh$ and $\rKh$ be the unreduced and
reduced Khovanov homology at a marked point $b$, with
$\rKh^{0,0}(U;R)=R$.
The superscripts $i,j$ denote the homological and quantum gradings,
respectively. We also use $\rKh^i=\oplus_j \rKh^{i,j}$ and $\Kh^i=\oplus_j \Kh^{i,j}$.
For a nonsplit positive link $L$, set
\[
a(L)=1-\chi(L).
\]
For a fibered link with fiber $F$, this quantity is $b_1(F)$; for a knot, $a(K)=2g(K)$.

Following Ito's notation \cite{Ito22}, let $p(L)$ denote the number
of nontrivial prime factors of a nonsplit link, with $p(U)=0$.
For a split link, sum this quantity over its split factors.

In \cite[Theorems~1.2 and~1.5]{Manikandan26},
Manikandan computed the reduced Khovanov homology
of positive links in homological grading one and proved that it is
torsion-free in homological grading two over $\Z$. In the fibered case, we determine
the full group in homological grading two over $\Z$ and $\F_2$.

\begin{theorem}\label{kh:thm:fibered-all}
Let $L$ be a fibered positive link. For every marked point $b$,
\begin{equation}\label{kh:eq:main-reduced}
\rKh^2(L,b;R)\cong R^{p(L)}\{a(L)+4\}.
\end{equation}
A nontrivial fibered positive link is prime if and only if its
reduced Khovanov homology in homological grading two has rank one.
\end{theorem}

We next compute the unreduced homology in the same homological
grading. This is when coefficients make a difference.
Let $\ell(L)$ count the unordered partitions of the components
into two parts with total linking number one. Equivalently, for
$L=L_1\cup\cdots\cup L_\mu$,
\begin{equation}\label{kh:eq:linking-count}
\ell(L)=\#\left\{S\subseteq\{2,\ldots,\mu\}\;\middle|\;
\sum_{\substack{i\in S\\j\notin S}}
\lk(L_i,L_j)=1\right\}.
\end{equation}
In particular, $\ell(K)=0$ for a knot $K$.

\begin{theorem}\label{kh:thm:main-unreduced}
For a fibered positive link $L$ with $a=a(L)$,
\begin{equation}\label{kh:eq:main-unreduced}
\begin{aligned}
\Kh^2(L;\Z)&\cong
\Z^{p(L)}\{a+3\}\oplus\Z^{\ell(L)}\{a+5\},\\
\Kh^2(L;\F_2)&\cong
\F_2^{p(L)}\{a+3\}\oplus\F_2^{p(L)}\{a+5\}.
\end{aligned}
\end{equation}
For a knot $K$, the upper formula becomes
$\Kh^2(K;\Z)\cong\Z^{p(K)}\{2g(K)+3\}$.
\end{theorem}

Write $T_2(G)=\{x\in G:2x=0\}$ for the $2$-torsion subgroup of
$G$, as in \cite{KMMS}. The universal coefficient
theorem turns the difference between the two ranks
into the following statement.

\begin{corollary}\label{kh:cor:torsion-intro}
For a fibered positive knot $K$,
\[
T_2\bigl(\Kh^{3,2g(K)+5}(K;\Z)\bigr)
\cong(\Z/2)^{p(K)}.
\]
\end{corollary}

\subsection{Knot Floer homology}

For links in $S^3$, let $\HFK(L)$ denote the knot Floer homology
defined by Ozsv\'{a}th--Szab\'{o} \cite{OSknot}, and independently by Rasmussen \cite{Rthesis}.
For a prime positive braid knot $K$, the author and Hedden
\cite{CH} proved that $\HFK(K,g(K)-1)\cong\F[-1]$.
For prime fibered positive knots beyond positive braids, this group
need not have rank one, say, $10_{154}$ and $10_{161}$
\cite{Knotinfo}. Instead, we determine the entire Maslov-graded next-to-top term which answers the first half of
\cite[Question~4.4]{CH} affirmatively.

\begin{theorem}\label{thm:knots}
Let $K\subset S^3$ be a fibered positive knot of genus $g$. There is a
nonnegative integer $z(K)$ such that
\begin{equation}\label{eq:main-hfk-knots}
\HFK(K,g-1)\cong
\F^{\,z(K)}[0]\oplus\F^{\,p(K)}[-1].
\end{equation}
Consequently, a nontrivial $K$ is prime precisely when
$\dim\HFK_{-1}(K,g-1)=1$.
\end{theorem}

For a nontrivial fibered positive knot $K$ of genus $g$, the
remaining dimension $z(K)$ is determined by the penultimate
coefficient $a_{g-1}$ of the Alexander polynomial
$\Delta_K(t)=t^g+a_{g-1}t^{g-1}+\ldots$.
Since the graded Euler characteristic of knot
Floer homology is the Alexander polynomial \cite{OSknot,Rthesis},
Theorem~\ref{thm:knots} gives
\begin{equation}\label{eq:alexander-coefficient}
a_{g-1}(K)=z(K)-p(K).
\end{equation}
The Alexander polynomial and the prime decomposition determine
the entire next-to-top term:
\[
\HFK(K,g-1)\cong
\F^{\,a_{g-1}(K)+p(K)}[0]\oplus\F^{\,p(K)}[-1].
\]
For a prime knot, its dimension is $a_{g-1}(K)+2$.

For a prime knot in this class, $a_{g-1}$
also bounds the number of fixed points of
the monodromy. By \cite[Theorem~1.2]{Nifp2},
the monodromy is freely isotopic to a diffeomorphism with at most
$a_{g-1}(K)+1$ fixed points.

\begin{corollary}\label{cor:fixed-points}
Let $K$ be a prime fibered positive knot with $a_{g-1}(K)=-1$. Then its
monodromy is freely isotopic to a diffeomorphism without fixed points.
\end{corollary}

Our result extends to links. Following \cite{CH}, we define the
\emph{genus} of an oriented link $L$ by
\begin{equation}\label{eq:link-genus}
g(L)=\min_{\Sigma}\left\{\frac{|L|-\chi(\Sigma)}2\;\middle|\;\Sigma\text{ is a
(possibly disconnected) Seifert surface of }L\right\},
\end{equation}
where $|L|$ denotes its number of components. For a fibered link with fiber $F$ of ordinary
genus $h$, this can also be expressed as
\[
g(L)=\frac{|L|-\chi(F)}2=h+|L|-1=i(F).
\]
Here $i(F)$ is the surface index used in \cite{CHS22}, which
coincides with the highest Alexander grading of $\HFK(L)$.
The Khovanov parameter is translated as $a(L)=2g(L)-|L|+1$.
We note that a fibered link is necessarily nonsplit.

\begin{theorem}\label{thm:links}
Let $L\subset S^3$ be a fibered positive link. Then, for some
nonnegative integer $z(L)$,
\begin{equation}\label{eq:main-hfk-links}
\HFK(L,g(L)-1)\cong
\F^{\,z(L)}[0]
\oplus\F^{\,p(L)+|L|-1}[-1].
\end{equation}
\end{theorem}

\noindent We note that the same prime factor count $p(L)$ occurs in the two theories:
\begin{equation}\label{eq:common-count}
\rank\rKh^2(L,b;\Z)
=\dim\HFK_{-1}(L,g(L)-1)-|L|+1
=p(L).
\end{equation}
For knots, both homology ranks equal $p(K)$.

\subsection{Khovanov torsion}\label{subsec:torsion-applications}
Combining Corollary~\ref{kh:cor:torsion-intro} with the nonfibered
case of \cite[Theorem~1.1]{KMMS}, we have the following theorem.

\begin{theorem}\label{kh:thm:positive-torsion-main}
Every nontrivial positive knot has $2$-torsion in its
unreduced Khovanov homology with $\Z$-coefficients. Such torsion
occurs in homological grading two if the knot is nonfibered,
and in homological grading three if it is fibered.
\end{theorem}

Positive knots are semi-adequate. A knot is \emph{semi-adequate}
if it admits a diagram whose state graph at one of the two extreme
resolutions has no loops; see Subsection~\ref{kh:sec:coefficients}. Przytycki and Sazdanovi\'c
\cite{PrzytyckiSazdanovic14} proved the existence of $2$-torsion
when the adequate state graph contains a cycle of length at least
three. In the remaining case, the reduced state graph is a tree,
and the knot is fibered positive up to mirroring. We obtain the following result.

\begin{theorem}\label{kh:thm:semiadequate-torsion-main}
Every nontrivial semi-adequate knot has nonzero $2$-torsion in
its unreduced Khovanov homology with integral coefficients.
\end{theorem}

\subsection{Proof strategy}\label{subsec:comparison-strategy}
The two calculations share a geometric input: a prime fibered
positive diagram admits a leaf crossing whose oriented smoothing
remains prime and fibered, until the positive Hopf link is reached.
We choose this crossing using the interlacement graph. The Khovanov
and knot Floer arguments then separately control the term obtained
by changing this crossing.

Skein inductions for positive braids go back to Van Buskirk
\cite{Buskirk}. For a connected positive diagram of a fibered link,
the reduced Seifert graph is a tree. Stoimenow used this fact to
apply a skein induction \cite{Stoimenow05}.
Buchanan likewise simplified crossing multiplicities in reduced Seifert trees
\cite[Definition~2.12]{Buchanan22}.
Our interlacement graph $\mc I_D$ encodes the mixed condition of
Dasbach--Lin \cite[Section~4]{DasbachLin06} and records the
information needed to preserve primeness.

Both proofs proceed by induction using skein exact sequences. However, in contrast to
the positive braid induction, see
Remark \ref{rmk:pbinduction} and \cite{CH}, the skein sequence argument alone is not enough
to determine the full group we desire, because we have less control on the
crossing changed link. We overcome this problem by an explicit primitive cycle construction for
Khovanov homology and a surface Murasugi decomposition for knot Floer homology.

\subsection*{Organization}
Section~\ref{sec:resolution} establishes the diagrammatic
factorization and preferred smoothing.
Section~\ref{kh:sec:preliminaries} recalls the Khovanov conventions
and oriented skein triangle. Section~\ref{kh:sec:fibered-calculation}
proves the reduced calculation, and
Section~\ref{kh:sec:unreduced} gives the unreduced groups and
the semi-adequate torsion theorem.

Section~\ref{sec:knot-floer} reviews the knot Floer conventions
and skein maps, then develops the filtered argument and graded
lower bound needed for the induction.
Section~\ref{sec:crossing-change} proves Maslov concentration for
the top term after the preferred crossing change by reducing the
generalized Seifert surface to a pretzel surface.
Section~\ref{sec:theorem-proofs} combines these results in a skein
induction to prove the next-to-top formulas.

\subsection*{AI Declaration}
The author used OpenAI’s ChatGPT to discuss mathematical ideas and to assist
with English writing, figure preparation, and literature checking.
ChatGPT identified a gap in an earlier argument for the preferred crossing.
This led to the interlacement graph argument in
Section~\ref{sec:resolution}. The author then reviewed and edited
the manuscript and takes full responsibility
for the mathematical statements, proofs, and references.

\begingroup
\emergencystretch=3em
\subsection*{Acknowledgements}
The author thanks Dongtai He, Matthew Hedden, and Yi Ni for helpful discussions.
The author also thanks Robert Lipshitz for helpful comments on an earlier draft.
The author was supported by NSFC grant No. 12401087.
\par
\endgroup

\section{Topology of fibered positive links}\label{sec:resolution}

In a positive diagram of a prime fibered link, we choose a crossing
whose oriented smoothing gives a prime fibered positive link.
An arbitrary crossing need not preserve primeness, see Figure~\ref{fig:preferred-crossing-example},
even if the crossing
is at a leaf of the reduced Seifert tree that we will introduce momentarily.
We first recover the prime pieces from the cyclic order at the
Seifert circles, then reduce each piece to a Hopf link. Both the
Khovanov and knot Floer calculations use this reduction.

\begin{figure}[htbp]
\centering
\includegraphics[width=\linewidth]{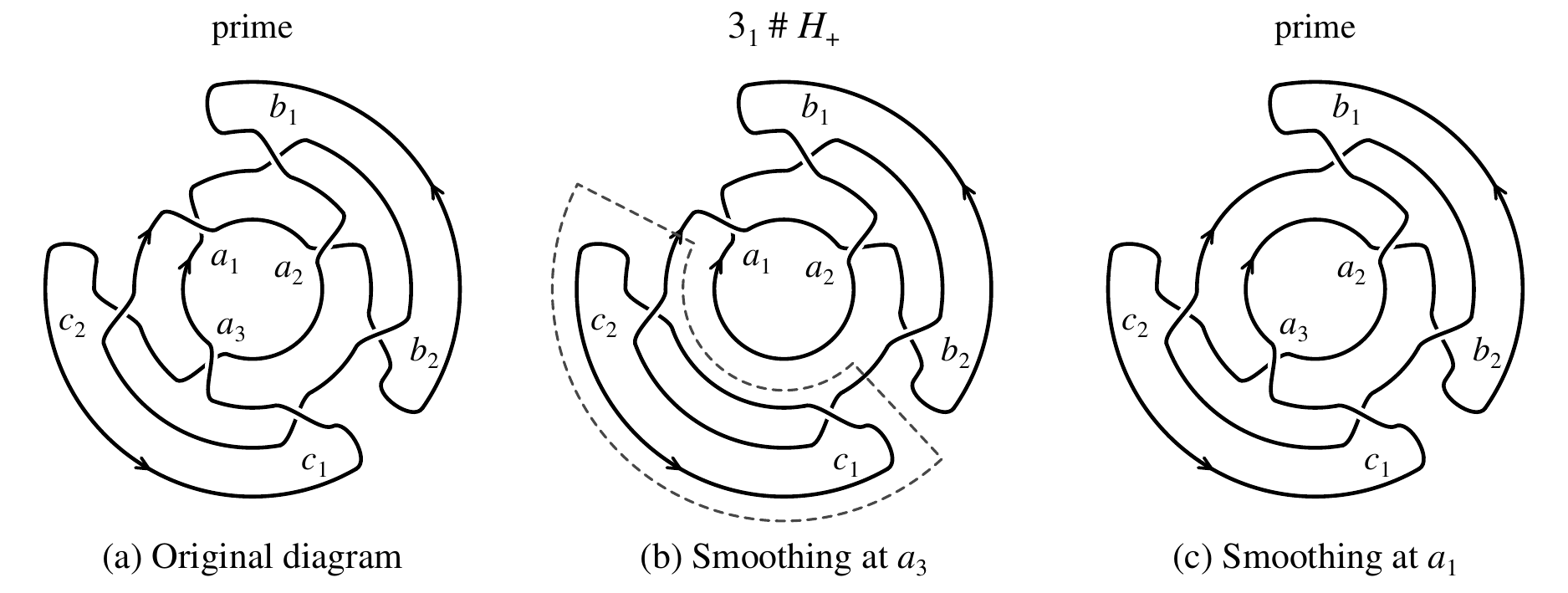}
\caption{Smoothing $a_3$ in the prime fibered positive diagram (a)
gives the composite link $3_1\#H_+$ in (b), where $H_+$ is the
positive Hopf link. Smoothing $a_1$ gives the prime link in (c).
The dashed curve in (b) exhibits the connected sum.}
\label{fig:preferred-crossing-example}
\end{figure}

\subsection{Interlacement and prime decomposition}\label{subsec2.1}
The \emph{Seifert graph} $G_D$ of an oriented diagram $D$ has a
vertex for each Seifert circle and an edge for each crossing. Its
\emph{reduced Seifert graph} $T_D$ replaces parallel edges by a single
edge. We call each collection of parallel edges a \emph{crossing bundle},
and its size the \emph{multiplicity} of the corresponding edge of $T_D$.

By Stoimenow's characterization \cite[Proof of Corollary 1]{Stoimenow05},
a connected positive diagram represents a fibered link exactly when
$T_D$ is a tree; its canonical Seifert surface is a fiber surface.
If the diagram is reduced, meaning that it has no nugatory crossings,
every crossing bundle has multiplicity at least two, otherwise it gives a nugatory crossing.
A nontrivial prime fibered positive link admits such a reduced
diagram, which is prime by \emph{visual primeness} \cite{Ozawa}.

Let $D$ be a connected reduced positive diagram with at least one
crossing and reduced Seifert tree $T$. Its \emph{interlacement graph}
$\mc I_D$ has vertex set $E(T)$. Two vertices $e,f$ are adjacent
when their bundles have crossings in cyclic order $e,f,e,f$ around
a common Seifert circle.

\begin{lemma}\label{lem:interlacement}
The diagram $D$ is prime if and only if $\mc I_D$ is connected.
\end{lemma}

\begin{proof}
A decomposing circle determines a two-edge cut of the projection
graph. Each Seifert circle crosses the cut an even number of times,
so both cut edges lie on one circle $u$. Every other circle lies
entirely on one side. Connectivity along each branch of $T-u$
puts all occurrences of its attaching bundle on the same
side. The cut selects a proper consecutive interval of $u$ containing
all occurrences of a nonempty proper collection of incident bundles.
Conversely, such an interval assigns each whole branch to one side
and gives a planar two-edge cut, realized by a decomposing circle.

For the bundles incident to $u$, such an interval exists exactly
when their interlacement graph is disconnected. One direction is
immediate. For the other, draw the convex hull of the occurrences
of each bundle on a circle. Two hulls intersect exactly when the
bundles interlace. The components form disjoint connected unions of
hulls, whose boundary points give a noncrossing partition. An
innermost block gives the required interval.

Finally, $\mc I_D$ is connected exactly when all these local graphs
are connected. A path leaving the bundles incident to $u$ can return
only through the same bundle, since $T$ is a tree, so it cannot
connect distinct local components. Conversely, connected local
graphs join all bundles along $T$.
\end{proof}

The cyclic order is essential. The positive braids
$\sigma_1^2\sigma_2^2$ and $\sigma_1\sigma_2\sigma_1\sigma_2$
have the same reduced Seifert tree, with two edges $e,f$ of
multiplicity two. Their orders at the common circle are $e,e,f,f$
and $e,f,e,f$, respectively. The first interlacement graph is
disconnected and the second connected, so the first diagram is
composite and the second prime.

\begin{lemma}\label{lem:diagram-decomposition}\label{prop:interlacement-count}
A connected reduced positive diagram $D$ of $L$ with reduced Seifert
graph a tree decomposes along diagrammatic connected sums into prime
fibered positive diagrams. Each factor is obtained by separating consecutive collections of
branches at Seifert circles. The nontrivial factors correspond
to the connected components of $\mc I_D$.
In particular,
\begin{equation}\label{eq:diagram-factor-count}
p(L)=b_0(\mc I_D),
\end{equation}
with both sides zero for the crossingless unknot diagram.
\end{lemma}
\begin{proof}
If $D$ is not prime, cut along the decomposing circle supplied by
Lemma~\ref{lem:interlacement}. Each piece has an assigned subtree
as its reduced Seifert graph, with a copy of the separating Seifert
circle. The bundles retain their multiplicities and cyclic orders,
so the pieces remain reduced and fibered positive. Bundles in
different pieces do not interlace. The graph $\mc I_D$ is the disjoint
union of their induced interlacement graphs.

Repeat until every piece is prime. Each cut decreases the number
of crossings in either piece, and no bundle is lost.
By Lemma~\ref{lem:interlacement}, the final interlacement graphs
are the connected components of $\mc I_D$.
Visual primeness \cite{Ozawa} gives the claimed prime decomposition
and count, independently of the cuts.
\end{proof}

\subsection{The preferred smoothing}

We choose a leaf crossing whose smoothing preserves primeness.

\begin{lemma}\label{lem:prime-smoothing}
Let $D$ be a connected prime positive diagram with at least one
crossing, no nugatory crossings, and reduced Seifert graph a tree.
There is a crossing incident to a leaf whose oriented smoothing is
prime after removing at most one nugatory crossing, or is the unknot.
If the link represented by $D$ is not the positive Hopf link, the
smoothed link is nontrivial, prime, fibered, and positive.
\end{lemma}

\begin{proof}
Write $T$ for the reduced Seifert tree. Suppose it has at least two
edges. By Lemma~\ref{lem:interlacement}, $\mc I_D$ is connected.
Choose a leaf $e$ of a spanning tree of $\mc I_D$, so that
$\mc I_D-e$ remains connected. Then $e$ is a leaf edge of $T$:
otherwise the bundles on opposite sides of $T-e$ would be
disconnected in $\mc I_D-e$.

Write $e=uv$, where $v$ is a leaf, and let $m\ge2$ be its
multiplicity. If $m=2$, smooth either crossing and remove the
remaining nugatory crossing at $v$. The resulting reduced Seifert
tree is $T-v$ and its interlacement graph is $\mc I_D-e$.
If $m\ge3$, retain two $e$-crossings witnessing interlacement with
a neighbor $f$, and smooth any other $e$-crossing. The new
interlacement graph contains $\mc I_D-e$ and the edge $ef$.
In either case it is connected, so the diagram is prime by
Lemma~\ref{lem:interlacement}.

If $T$ has one edge, smoothing reduces its multiplicity from $m$
to $m-1$. For $m\ge3$ the interlacement graph still has one vertex,
so the smoothed diagram is prime. For $m=2$, the original link is
the positive Hopf link and its smoothing is the unknot.

The resulting diagram remains connected and positive with reduced
Seifert graph a tree, so its canonical surface $F_0$ is a fiber
surface. Smoothing removes one band, and removing a nugatory
crossing preserves the surface. Thus
\[
b_1(F_0)=\sum_{e\in E(T)}\bigl(m(e)-1\bigr)-1.
\]
This vanishes only for a single bundle of two crossings, the Hopf
link. Otherwise the smoothed link is nontrivial and is prime by
\cite{Ozawa}.
\end{proof}

Let $D$ be a connected reduced positive diagram of a nontrivial
prime fibered link $L$ with fiber surface $F$, and put $N=b_1(F)$.
Iterating Lemma~\ref{lem:prime-smoothing} gives a sequence
\begin{equation}\label{eq:prime-reduction}
D=D_N\longrightarrow D_{N-1}\longrightarrow\cdots
\longrightarrow D_1\longrightarrow D_0.
\end{equation}
Each step is an oriented smoothing at a leaf of the reduced
Seifert tree, followed by removal of at most one nugatory crossing.
For $1\leq n\leq N$, the diagram $D_n$ represents a nontrivial prime
fibered positive link whose fiber has first Betti number $n$.
Since $b_1=1$ forces a single bundle of two crossings, $D_1$ is a
positive Hopf diagram; $D_0$ is the unknot diagram.
The intermediate diagrams may depend on the choices.

For positive braid closures, Van Buskirk's reduction \cite{Buskirk}
uses a square $\sigma_k^2$ and gives
\begin{equation}\label{eq:braid-square}
L_+=\widehat{\sigma_k^2\beta'},\qquad
L_0=\widehat{\sigma_k\beta'},\qquad
L_-=\widehat{\beta'}.
\end{equation}
The other two links are positive braid closures, possibly split.
Stoimenow's smoothing induction for fibered positive diagrams treats
the almost positive crossing-changed term separately
\cite[Proof of Corollary~1]{Stoimenow05}.
Lemma~\ref{lem:prime-smoothing} also chooses the smoothing to preserve
primeness. Intervening attachments at a general leaf bundle can
obstruct the Reidemeister II cancellation, so we calculate the
crossing-changed term separately below.

\section{A review of Khovanov homology}\label{kh:sec:preliminaries}

We recall the Khovanov homology of \cite{K00} in
Bar-Natan's formulation \cite{BN02}, the oriented
skein triangle, and Lee's deformation. Coefficients are $R=\Z,\F_2$,
except for Lee's deformation, where we use $\Q$ instead.

\subsection{The reduced and unreduced complexes}
Let $D$ be an oriented diagram with $n_+$ positive and $n_-$
negative crossings, ordered from $1$ to $n$. For crossings $x,y$,
write $x<y$ when $x$ precedes $y$ in this order. We use the resolutions
\[
\vcenter{\hbox{\includegraphics[height=1.4em]{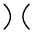}}}
\xleftarrow{\;0\;}
\vcenter{\hbox{\includegraphics[height=1.4em]{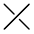}}}
\xrightarrow{\;1\;}
\vcenter{\hbox{\includegraphics[height=1.4em]{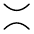}}}.
\]
A vertex $v\in\{0,1\}^n$ specifies a complete resolution with $|v|$ $1$-resolutions
and $k(v)$ circles. The oriented resolution is $0$ at a positive
crossing and $1$ at a negative crossing. Put
$A_R=R[X]/(X^2)$, with $1$ and $X$ in quantum gradings $1$ and $(-1)$, respectively. The unreduced complex is
\begin{equation}\label{kh:eq:cube-complex}
C^i(D;R)=\bigoplus_{|v|-n_-=i}
A_R^{\otimes k(v)}\{|v|+n_+-2n_-\},
\end{equation}
generated by enhanced states with each resolution
circle labelled $1$ or $X$. Edge maps use multiplication $m$ for
merges and the coproduct
\[
\Delta(1)=1\otimes X+X\otimes1,\qquad \Delta(X)=X\otimes X
\]
for splits, tensored with the identity on other circles.
The edge changing $v_i=0$ to $1$ has sign
$(-1)^{v_1+\cdots+v_{i-1}}$. Their sum is a differential of bigrading
$(1,0)$, and its homology $\Kh^{i,j}(L;R)$ is an oriented link
invariant.

Mark a point $b$ in $D$, and let $X_b$ multiply its circle label
by $X$. The reduced complex is defined as
$\rC(D,b;R)=(C(D;R)/X_bC(D;R))\{-1\}$, with homology $\rKh(L,b;R)$.
The marked circle has label $1$, and a marked split sends
$1$ to $1\otimes X$, with the first factor marked.

In particular, the unknot has
$\Kh^0(U;R)=R\{-1\}\oplus R\{1\}$ and $\rKh^0(U;R)=R$;
all its other groups vanish. Multiplication by $X_b$ identifies
the quotient with the image, with a shift, and gives
\begin{equation}\label{kh:eq:reduction}
0\longrightarrow\rC(D,b;R)\{-1\}\longrightarrow C(D;R)
\longrightarrow\rC(D,b;R)\{1\}\longrightarrow0.
\end{equation}
Over $\F_2$, the induced exact sequence splits \cite[Corollary~3.2.C]{Shumakovitch14}
\begin{equation}\label{kh:eq:mod-two-splitting}
\Kh(L;\F_2)\cong
\rKh(L,b;\F_2)\{-1\}\oplus\rKh(L,b;\F_2)\{1\}.
\end{equation}
We generally suppress the mark in the case that
our reduced homology statements hold for every choice.

For a disjoint union marked on $J_1$, the cube construction gives
\begin{equation}\label{kh:eq:disjoint}
\rC(J_1\sqcup J_2;R)\cong\rC(J_1;R)\otimes_R C(J_2;R).
\end{equation}

For a connected positive diagram of a fibered link $L$ with $c$
crossings and $s$ Seifert circles,
\begin{equation}\label{kh:eq:a}
c-s+1=a=a(L)=1-\chi(L).
\end{equation}
A reduced state with $r$ $1$-resolutions, $k$ circles, and $t$
unmarked circles labelled $1$ has quantum grading
\begin{equation}\label{kh:eq:grading}
j=c+r-k+1+2t.
\end{equation}
Combining \cite[Section~6]{Khovanov03},  \cite[Section~7]{Manikandan26}, and
\cite[Theorem~1.1]{KMMS}, we have the following result of Khovanov homology in low homological gradings.
\begin{lemma}\label{kh:lem:initial}
For a fibered positive link, every marking, and $R=\Z,\F_2$,
\begin{equation}\label{kh:eq:positive-inputs}
\rKh^0(L;R)=R\{a\},\qquad \rKh^1(L;R)=0.
\end{equation}
Moreover, $\rKh^2(L;\Z)$ is free, and
$\Kh^0(L;R)=R\{a-1\}\oplus R\{a+1\}$, $\Kh^1(L;R)=0$.
\end{lemma}

\subsection{The oriented skein triangle}
We concatenate the unoriented skein triangles as in Hedden--Watson
\cite[Section~3.1]{HW18}, following Rasmussen
\cite[Section~7.3]{Rasmussen2015}. For an oriented skein triple
$(L_+,L_-,L_0)$ with marked point $b$ outside the skein region, put
$C_\pm=\rC(L_\pm,b;R)$ and $C_0=\rC(L_0,b;R)$.
Let $C_\infty$ be the reduced complex of the other resolution,
with gradings normalized using the crossing signs of $L_\pm$
outside the skein region.
Resolving the chosen crossing in Khovanov's cube construction
\cite[Section~4.2]{K00} gives the unoriented crossing sequences of reduced complexes
\begin{align}
0&\longrightarrow C_\infty[1]\{2\}\longrightarrow C_+
\longrightarrow C_0\{1\}\longrightarrow0,\label{kh:eq:positive-crossing}\\
0&\longrightarrow C_0\{-1\}\longrightarrow C_-
\longrightarrow C_\infty[-1]\{-2\}\longrightarrow0.
\label{kh:eq:negative-crossing}
\end{align}

\begin{lemma}\label{kh:lem:cone}
There is an exact triangle
\begin{equation}\label{kh:eq:triangle}
\begin{gathered}
\cdots\longrightarrow C_-[2]\{4\}\longrightarrow C_+\longrightarrow E
\longrightarrow C_-[1]\{4\}\longrightarrow\cdots,\\
E=\Tot[C_0\{1\}\xrightarrow{\theta}C_0\{3\}],
\end{gathered}
\end{equation}
where the displayed columns of $E$ have local homological gradings $0,1$.
The unshifted map $\theta$ has bigrading $(0,-2)$.
\end{lemma}
\begin{proof}
Shift \eqref{kh:eq:negative-crossing} by $[2]\{4\}$ and compose its
projection with the inclusion in \eqref{kh:eq:positive-crossing}.
This gives $C_-[2]\{4\}\to C_+$, equal to the identity on their
common $C_\infty[1]\{2\}$ term. Suppressing internal differentials,
the cone has local part
\[
C_0\{1\}\oplus C_\infty\{2\}
\xrightarrow{\left(\begin{smallmatrix}s&1\\0&-t\end{smallmatrix}\right)}
C_\infty\{2\}\oplus C_0\{3\},
\]
where $s,t$ are the saddle maps. Canceling the identity pair leaves
$\theta=t\circ s$, up to an overall cone sign, with the stated shifts.
\end{proof}

The unsigned composite is $\Delta m=X\otimes1+1\otimes X$ on
two circles and $m\Delta=2X$ on one; here $X$ acts by multiplication.
The signs are inherited from the cone. Both saddles commute with
$X_b$, so the triangle descends to the reduced quotient for every
marking. When a participating circle is marked, $\Delta m$ reduces
to multiplication by $X$ on the other circle, and $m\Delta$ to zero.
The triangle and the two columns of $E$ give, respectively,
\begin{equation}\label{kh:eq:triangle-les}
\begin{aligned}
\cdots&\longrightarrow\rKh^{i-2,j-4}(L_-;R)
\longrightarrow\rKh^{i,j}(L_+;R)\longrightarrow H^{i,j}(E)\\
&\longrightarrow\rKh^{i-1,j-4}(L_-;R)\longrightarrow\cdots,
\end{aligned}
\end{equation}
and
\begin{equation}\label{kh:eq:E-sequence}
\begin{gathered}
0\longrightarrow \coker(\theta_{i-1})
\longrightarrow H^i(E)
\longrightarrow \ker(\theta_i)\longrightarrow0,\\
\theta_i=\theta_*:H^i(C_0)\{1\}\longrightarrow H^i(C_0)\{3\}.
\end{gathered}
\end{equation}
Both sequences preserve shifted quantum grading and apply whether
the smoothing gains or loses a component. For split unions, we use
\eqref{kh:eq:disjoint}.

\begin{remark}\label{kh:rem:saddle-map}
We do not determine the maps $\theta_i$ by definition.
When $L_0$ is fibered positive, the quantum concentration result
in Section~\ref{kh:sec:fibered-calculation} and the grading shift
give $\theta_0=\theta_2=0$, while $H^1(C_0)=0$.
The sequences \eqref{kh:eq:E-sequence} and \eqref{kh:eq:triangle-les}
then determine the required groups and skein maps from adjacent
vanishing terms, without an explicit calculation of $\theta_*$.
\end{remark}

\subsection{Lee's deformation}\label{kh:sec:lee}
For the unreduced calculation in Subsection~\ref{kh:subsec:unreduced-calculation},
we use Lee's filtered deformation over $\Q$, with differential
$d_{\mathrm{Lee}}=d+\Phi$, where $d$ and $\Phi$ have bigradings
$(1,0)$ and $(1,4)$, respectively. Its spectral sequence starts from
$\Kh^{i,j}(L;\Q)$ and has subsequent differentials of bigradings
$(1,4k)$, $k\ge1$ \cite[Section~3]{Rasmussen2010}. It converges to
\[
E_\infty^{i,j}\cong\operatorname{gr}_q^j H^i_{\mathrm{Lee}}(L),
\]
the associated graded of Lee homology for its quantum filtration.

Lee homology has a basis indexed by the orientations of the
underlying link. Relative to the given orientation of
$L=L_1\cup\cdots\cup L_\mu$, reversing the components indexed by
$S\subseteq\{1,\ldots,\mu\}$ gives a generator in homological grading
\[
i(S)=2\sum_{\substack{r\in S\\t\notin S}}\lk(L_r,L_t).
\]
Complementary subsets give opposite orientations in the same
homological grading, so \cite[Proposition~4.3]{Lee05} gives
\begin{equation}\label{kh:eq:lee-orientations}
\dim H^i_{\mathrm{Lee}}(L)
=2\#\{S\subseteq\{2,\ldots,\mu\}:i(S)=i\}.
\end{equation}
In particular, Lee homology of a knot is concentrated in
homological grading zero, and
$\dim H^2_{\mathrm{Lee}}(L)=2\ell(L)$, with $\ell(L)$ as in
\eqref{kh:eq:linking-count}.

The differential $d_{\mathrm{Lee}}$ preserves quantum grading
modulo four. Since every enhanced state has quantum grading
congruent to $\mu$ modulo two, the residue classes are $\bar\mu$
and $\overline{\mu+2}$. Write $H^{i,\bar j}_{\mathrm{Lee}}(L)$ for
these summands, we have

\begin{lemma}\cite[Lemma~3.5]{Rasmussen2010}\label{kh:lem:lee-residues}
For every oriented $\mu$-component link and every $i$,
\[
\dim H^{i,\bar\mu}_{\mathrm{Lee}}(L)
=\dim H^{i,\overline{\mu+2}}_{\mathrm{Lee}}(L)
=\frac12\dim H^i_{\mathrm{Lee}}(L).
\]
In homological grading two, the dimensions above are $\ell(L)$.
\end{lemma}

\section{Reduced Khovanov homology of fibered positive links}\label{kh:sec:fibered-calculation}

By Lemma~\ref{lem:interlacement}, every leaf bundle in a prime
diagram with at least two bundles interlaces another bundle.
Lemma~\ref{lem:prime-smoothing} chooses a smoothing that stays
prime, fibered, and positive until it finally reaches the Hopf link.
\subsection{The calculation at a leaf crossing}\label{kh:sec:leaf}

Let $D_+$ be a connected positive diagram with reduced Seifert
graph a tree $T$, and set $a=c-s+1$ as in \eqref{kh:eq:a}. Choose a leaf $v$ with
neighbor $u$ and multiplicity $q\geq2$. Change one of the $uv$
crossings, denoted as $c_0$, to obtain $D_-$. Let $P$ be the set of
the remaining $q-1$ crossings in this bundle.
For sign purposes, we choose a crossing order such that $c_0$ is first.

\begin{proposition}\label{kh:prop:leaf}
Suppose $q\geq3$, or $q=2$ and some neighbor of $u$ interlaces
the leaf $v$. For every marked point and $R=\Z,\F_2$,
\[
\rKh^0(L_-;R)=R\{a-2\},\qquad
\rKh^{-1}(L_-;R)=0.
\]
\end{proposition}
\begin{proof}
Since $D_-$ has one negative crossing, a state with $r$
$1$-resolutions has homological grading $r-1$. There are no
terms below homological grading $-1$, so we only need to compute
$\ker d^{-1}$ and $\ker d^0/\im d^{-1}$.

\smallskip
\noindent\emph{Step 1: The states with low homological gradings.}
We first put the marking on the leaf circle $v$. A state in homological grading zero has
exactly one $1$-resolved crossing. If $t$ unmarked circles have
label $1$, its circle count and quantum grading are
\[
\begin{array}{c|cc}
\text{$1$-resolved crossing}&\text{circles}&\text{quantum grading}\\ \hline
c_0\text{ or }e\in P&s&a-2+2t\\
f\text{ outside the leaf bundle}&s-2&a+2t.
\end{array}
\]
The all-$0$ resolution merges $u,v$. Its differential to the
oriented resolution splits off $u$ with label $X$, so this
component is a split injection and $\rKh^{-1}(L_-;R)=0$.

In homological grading zero, write a cycle as the sum \[
y=\sum_e y_e
\] of terms with
only $e$ being $1$-resolved. Subtracting a boundary, we choose the unique representative
for which every state in $y_{c_0}$ labels $u$ by $1$.

For any pair of crossings, the two incoming cube signs are
opposite, so the unsigned saddle images of the corresponding
components of $y$ agree. At $(c_0,e)$, $e\in P$, both saddles
are marked merges. The states in $y_e$ whose new unmarked circle
has label $1$ have the same coefficients as the corresponding
states in $y_{c_0}$.

\smallskip
\noindent\emph{Step 2: The states in quantum grading $a$.}
Let $f$ denote a crossing outside the $uv$ bundle.
If $f$ joins $u$ to $w$,
the saddle from $y_f$ is a marked split, so its new
unmarked circle has label $X$. The saddle from $y_{c_0}$ merges
the $1$-labelled $u$ with $w$. The label at $w$ must be $X$ in
$y_{c_0}$, and this equation determines $y_f$.

For $f$ between other circles, the two images instead have labels
$1$ and $X$ on $u$. Both vanish, giving $y_f=0$ and zero merge
of $y_{c_0}$ along $f$. Root the tree $T-v$ at $u$. Once a parent has
label $X$, this zero merge forces its child to have label $X$:
$m(X\otimes1)=X$, whereas $m(X\otimes X)=0$. Induction leaves
only the state in $y_{c_0}$ with $u$ labelled $1$ and every other
unmarked circle labelled $X$; write its coefficient as $\lambda$.
For $f$ incident to $u$, $y_f$ is $\lambda$ times the state with
all unmarked labels $X$. These states all have quantum grading $a$.

\smallskip
\noindent\emph{Step 3: The homology with the leaf $v$ marking.}
If $q\ge3$, choose distinct $e,f\in P$. Their two new unmarked
circles are separate or nested. In the separate case, the two
saddle images of the $1$-labelled states place $1$ on different
circles. In the nested case, one image contains both placements,
by $\Delta(1)=1\otimes X+X\otimes1$, while the other contains
only one. In either case, an unmatched coefficient gives
$\lambda=0$. Comparing the states with both new circles labelled
$X$ also equates the remaining label coefficients at $e$ and $f$.

If $q=2$, write $P=\{e\}$. For a neighbor $w$ of $u$ other than
$v$, let $\beta_w$ be the coefficient in $y_e$ of the state with
$w$ labelled $1$ and every other unmarked circle labelled $X$.
The equation at $(e,f)$ for a crossing $f$ from $u$ to $w$ gives
\[
\beta_w=
\begin{cases}
\lambda,&f\text{ attaches to the marked circle},\\
0,&f\text{ attaches to the unmarked circle}.
\end{cases}
\]
Indeed, the first merge contributes $\beta_w$ and the second
$\lambda+\beta_w$, while the image from $y_f$ contributes
$\lambda$ in either case. An interlacing neighbor supplies both
attachments, forcing $\lambda=0$; see
Figure~\ref{kh:fig:leaf-attachments}.

\begin{figure}[htbp]
\centering
\includegraphics[width=.88\linewidth]{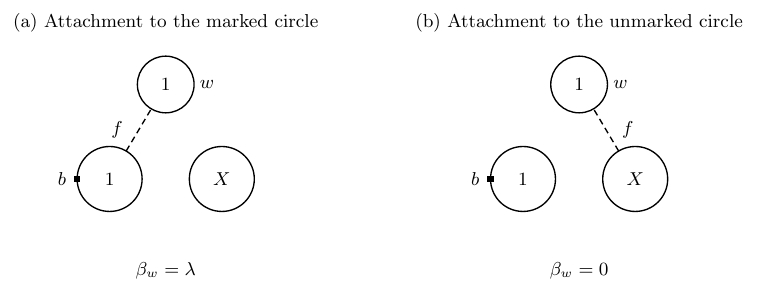}
\caption{The two attachments at $f$ after resolving $e$. Both
panels show the state with coefficient $\beta_w$. The square is
the mark.}
\label{kh:fig:leaf-attachments}
\end{figure}

Only components indexed by $P$ remain, with $X$ on the new
unmarked circle. Comparing with crossings outside the leaf bundle
forces every other neighbor of $u$ to have label $X$. The same
argument propagates this along $T-v$. Only the state with all
unmarked labels $X$ remains at each $e\in P$,
with a common coefficient. Let $\eta$ be their sum. It has
bigrading $(0,a-2)$ and is a cycle: parallel terms cancel, while
all other merges have two $X$ labels or merge $X$ into the marked
circle. Every normalized cycle is a multiple of $\eta$, and no
nonzero multiple is a boundary. Hence
$\rKh^0(L_-;R)=R\{a-2\}$ for the leaf marking.

\smallskip
\noindent\emph{Step 4: The homology with arbitrary marking.}
For any marking, the component of $d^{-1}$ at the oriented
resolution is $\Delta$ or a reduced split, both split injections
over $\Z$. The map $d^{-1}$ is split injective,
$\rKh^{-1}(L_-;R)=0$, and $\rC^0(D_-;\Z)/\im d^{-1}$ is free.
Its subgroup $\rKh^0(L_-;\Z)$ is therefore free.
The splitting in characteristic two \eqref{kh:eq:mod-two-splitting}
and the calculation with the leaf marked give
$\Kh^0(L_-;\F_2)=\F_2\{a-3\}\oplus\F_2\{a-1\}$.
Applying the same splitting at any other mark gives
$\rKh^0(L_-;\F_2)=\F_2\{a-2\}$.
In each quantum grading the universal coefficient sequence is
\[
0\longrightarrow\rKh^{0,j}(L_-;\Z)\otimes\F_2
\longrightarrow\rKh^{0,j}(L_-;\F_2)
\longrightarrow
\operatorname{Tor}_1^{\Z}(\rKh^{1,j}(L_-;\Z),\F_2)
\longrightarrow0.
\]
The first map, together with freeness, shows that the integral
group in homological grading zero is either zero or $\Z\{a-2\}$.
If it were zero, the reduction sequence would force integral
$\Kh^0(L_-)$ to vanish. This contradicts the injection
$\Z\{a-3\}\hookrightarrow
\Kh^0(L_-;\Z)$ obtained from the leaf marking.
\end{proof}

\begin{lemma}\label{kh:cor:split-groups}
Suppose $q=2$ and no neighbor of $u$ interlaces $v$. Then
$L_-=J_1\sqcup J_2$, where each $J_i$ is fibered
and positive, possibly the unknot. For every marking,
\[
\rKh^0(L_-;R)=R\{a-2\}\oplus R\{a\},\qquad
\rKh^1(L_-;R)=0.
\]
\end{lemma}
\begin{proof}
The two $uv$ attachments divide $u$ into arcs. Noninterlacement
and the tree condition put each other branch on one side. After
the crossing change, one sublink passes over the other at both
distinguished crossings. Pulling it above the other removes the
pair and separates positive diagrams whose reduced Seifert graphs
are trees. Both sublinks are fibered positive.
With $a_i=a(J_i)$, deleting two crossings and preserving the total
number of Seifert circles gives $a_1+a_2=(c-2)-s+2=a-1$.

Mark $J_1$ and use \eqref{kh:eq:disjoint}. The tensor product in homological grading zero is
$R\{a_1\}\otimes(R\{a_2-1\}\oplus R\{a_2+1\})$.
The group in homological grading one vanishes by Lemma~\ref{kh:lem:initial} and the K\"unneth
formula. The same argument applies to a marking on $J_2$.
\end{proof}

\subsection{The skein argument}\label{kh:sec:main-proofs}

We first prove quantum concentration without a primeness assumption,
then use the resulting injection and the preferred smoothing to
bound the rank by one for prime links.

\begin{theorem}\label{kh:thm:concentration}
For a fibered positive link $L$ and any marked component,
\[
\rKh^{2,j}(L;R)=0\quad\text{if }j\ne a(L)+4,
\qquad R=\Z,\F_2.
\]
Moreover, the group with $\Z$-coefficients is free.
\end{theorem}
\begin{proof}
The freeness statement follows from Lemma~\ref{kh:lem:initial}.
We prove concentration by induction on $a=a(L)$, starting with
the unknot when $a=0$. Choose a reduced connected positive diagram;
its reduced Seifert graph is a tree and every
edge has multiplicity at least two. Smooth one crossing at any leaf.
The smoothing is again fibered positive, with $a(L_0)=a-1$.
Use an arbitrary marking, inherited by both $L_-$ and $L_0$
outside the skein disk.
Set $C_0=\rC(L_0;R)$ and use the triangle of Lemma~\ref{kh:lem:cone}
\[
\rC(L_-;R)[2]\{4\}\longrightarrow\rC(L_+;R)
\longrightarrow E,\qquad
E=\Tot[C_0\{1\}\xrightarrow{\theta}C_0\{3\}].
\]
The unshifted map $\theta$ shifts quantum grading by $(-2)$. By induction and
\eqref{kh:eq:positive-inputs}, it is zero on $H^0(C_0)$ and $H^2(C_0)$. Thus
\[
H^1(E)=R\{a+2\},\qquad H^2(E)=\rKh^2(L_0;R)\{1\}.
\]
The relevant exact sequence is
\begin{equation}\label{kh:eq:main-les}
\begin{aligned}
0&\longrightarrow R\{a+2\}
\xrightarrow{\alpha}\rKh^0(L_-;R)\{4\}
\longrightarrow\rKh^2(L_+;R)\\
&\longrightarrow\rKh^2(L_0;R)\{1\}
\xrightarrow{\beta}\rKh^1(L_-;R)\{4\}
\longrightarrow\ldots.
\end{aligned}
\end{equation}
In the cases of Proposition~\ref{kh:prop:leaf}, the target of $\alpha$ is
$R\{a+2\}$. Over $\F_2$, injectivity makes $\alpha$ an isomorphism.
Over $\Z$, its cokernel injects into the free group
$\rKh^2(L_+;\Z)$, so its nonzero scalar is a unit. Exactness gives
\begin{equation}\label{kh:eq:kernel}
\rKh^2(L_+;R)\cong\ker\beta
\subseteq\rKh^2(L_0;R)\{1\}.
\end{equation}
This proves concentration without computing $\rKh^1(L_-)$.

In the remaining case, Lemma~\ref{kh:cor:split-groups} applies. The map $\alpha$
is again an isomorphism onto the $a+2$ summand, and the final group in
\eqref{kh:eq:main-les} vanishes. We obtain
\begin{equation}\label{kh:eq:split-recursion}
0\longrightarrow R\{a+4\}\longrightarrow\rKh^2(L_+;R)
\longrightarrow\rKh^2(L_0;R)\{1\}\longrightarrow0.
\end{equation}
Both ends are supported in quantum grading $a+4$. This proves the induction.
\end{proof}

Now we combine the skein injection with the preferred smoothing
to bound the rank for prime links.

\begin{proposition}\label{kh:prop:prime-bound}
Every nontrivial prime fibered positive link satisfies
\[
\rank_R\rKh^2(L;R)\leq1
\]
for every marked component.
\end{proposition}
\begin{proof}
By moving the marked point along the arc in the diagram, we assume
the marked point is outside the skein region.
Iterating Lemma~\ref{lem:prime-smoothing} reduces $L$ to the positive Hopf link
through preferred smoothings of prime fibered positive links.
Before the Hopf endpoint, the chosen leaf has multiplicity at least
three, or has multiplicity two and interlaces another bundle by
primeness. Proposition~\ref{kh:prop:leaf} applies, and
\eqref{kh:eq:kernel} gives at each step an inclusion
\begin{equation}\label{kh:eq:prime-injection}
\rKh^2(L;R)\hookrightarrow\rKh^2(L_0;R)\{1\}.
\end{equation}
In the two-crossing Hopf diagram, the image from homological grading one spans
the $X$-labelled state of the sole unmarked circle in homological grading two.
The remaining $1$-labelled state gives
$\rKh^2(H_+;R)=R\{5\}$, proving the bound for the chosen marking.
\end{proof}

\subsection{A primitive class in homological grading two}\label{kh:sec:nonvanishing}

We proved that the rank of $\rKh^2(L;R)$ is either $0$ or $1$ for nontrivial prime fibered positive links.
We now show it must be $1$ by constructing an integral
cycle and a functional that vanishes on boundaries and takes value
one on the cycle. Its class is then \emph{primitive}, generating
a direct summand $\Z$.
Let $D$ be a connected reduced positive diagram with at least one
crossing, whose reduced Seifert graph is a tree and whose crossing
bundles have multiplicity at least two.

\begin{lemma}\label{kh:thm:cycle}
For every marked point there are $z\in\rC^{2,a+4}(D;\Z)$ and a
homomorphism $\lambda:\rC^{2,a+4}(D;\Z)\to\Z$ such that
\[
dz=0,\qquad\lambda d^1=0,\qquad\lambda(z)=1.
\]
Here $d^1:\rC^{1,a+4}(D;\Z)\to\rC^{2,a+4}(D;\Z)$ is the Khovanov differential.
Thus $[z]$ generates a direct summand $\Z$ in
$\rKh^{2,a+4}(D;\Z)$. With coefficients $R$, the same construction gives a summand $R$.
\end{lemma}
\begin{proof}
Since $D$ is positive, a generator in bigrading $(2,a+4)$ has
two $1$-resolved crossings. Write $D_S$ for the resolution with
precisely the crossings in $S$ $1$-resolved.

\smallskip
\noindent\emph{Step 1: The states in bigrading $(2,a+4)$.}
Two crossings in one
bundle merge its two Seifert circles and then split the resulting
circle. This leaves $s$ circles, so \eqref{kh:eq:grading} requires
exactly one unmarked circle labelled $1$. Crossings in different
bundles give two merges along distinct edges of the tree, leaving
$s-2$ circles. Then every unmarked label is $X$. The marked circle
always has label $1$, and we leave unmarked $X$-labels implicit.

To choose the unmarked circle labelled $1$ in the first case, take
the marked circle as the root of the Seifert tree. In each bundle the circle
nearer the root is the parent and the other is the child. Choose
a starting point, or cut, between attachments on each circle,
using the marked point on the root. Order the crossings bundle
by bundle, and within each bundle follow the oriented parent
circle from its cut.

For $x<y$ in the same bundle, let $B(x,y)$ be the circle of
$D_{\{x,y\}}$ that follows the oriented parent interval
from $x$ to $y$ and returns along the corresponding child interval.
Parallel crossing arcs have the same cyclic order on the two
oriented Seifert circles, so these intervals are well defined.
The circle $B(x,y)$ avoids the mark: at the root its parent
interval avoids the cut; away from the root the marked circle is unchanged. Let $u_{xy}$ be the enhancement of $D_{\{x,y\}}$
with $B(x,y)$ labelled $1$. For $x,y$ in different bundles, let
$v_{xy}$ be the enhancement of $D_{\{x,y\}}$ with all unmarked
circles labelled $X$. Both states have bigrading $(2,a+4)$.

\begin{figure}[htbp]
\centering
\includegraphics[width=.88\linewidth]{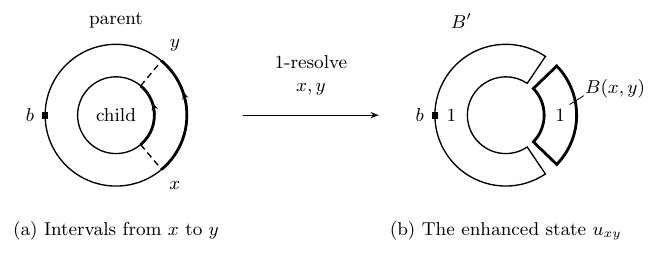}
\caption{The state $u_{xy}$. The heavier circle $B(x,y)$ follows
the parent interval from $x$ to $y$ and returns along the child
interval. It receives label $1$. In the illustrated root case,
the complementary circle $B'$ contains the mark. All omitted
unmarked circles have label $X$.}
\label{kh:fig:interval-circle}
\end{figure}

\smallskip
\noindent\emph{Step 2: The correction terms.}
Start with the sum of the $u_{xy}$ over pairs in the same bundle.
The differential terms involving three crossings in one bundle cancel.
Resolving a crossing $w$ outside the bundle of $x,y$ can still give
a nonzero term when it merges the $1$-labelled circle $B(x,y)$ with
another circle. We cancel these remaining terms by adding suitable
multiples of the $v$-states. To determine their coefficients, we record
whether $w$ meets $B(x,y)$ by setting
\[
\varepsilon(x,y;w)=
\begin{cases}
1,&\text{an attachment of $w$ lies on $B(x,y)$},\\
0,&\text{otherwise},
\end{cases}
\qquad\varepsilon(x,x;w)=0.
\]
This records membership in an interval on the shared Seifert
circle and vanishes when the bundles have no circle in common.

For $x<y<w$ in one bundle and $t<t'$ in another, cyclic intervals
give
\begin{align}
\varepsilon(x,w;t)&=\varepsilon(x,y;t)+\varepsilon(y,w;t),
\label{kh:eq:additive}\\
\varepsilon(x,y;t')-\varepsilon(x,y;t)
&=\varepsilon(t,t';x)-\varepsilon(t,t';y).
\label{kh:eq:reciprocity}
\end{align}
The first identity concatenates adjacent intervals. For the
second, membership in $(x,y)$ increases by one on passing $x$
and decreases by one on passing $y$. These arguments use cyclic
order and remain valid when a child interval crosses its chosen cut.

For bundles $F<G$, cancellation by the $v$-states requires
coefficients satisfying
\begin{align}
\kappa(x',y)-\kappa(x,y)&=-\varepsilon(x,x';y),\label{kh:eq:rowdiff}\\
\kappa(x,y')-\kappa(x,y)&=\varepsilon(y,y';x),\label{kh:eq:coldiff}
\end{align}
where $x<x'$ lie in $F$ and $y<y'$ lie in $G$.
Reciprocity \eqref{kh:eq:reciprocity} gives compatibility.
Using the first crossings $f\in F$ and $g\in G$ as reference
points, a solution is
\begin{equation}\label{kh:eq:t}
\kappa(x,y)=-\varepsilon(f,x;g)+\varepsilon(g,y;x),
\qquad x\in F,\quad y\in G.
\end{equation}
Additivity gives the column difference; reciprocity gives the
row difference after the $g$ terms cancel. Define
\begin{equation}\label{kh:eq:Z}
z=\sum_{\substack{x<y\\\text{same bundle}}}u_{xy}
\; +\!\sum_{\substack{x<y\\\text{different bundles}}}
\kappa(x,y)v_{xy}.
\end{equation}
For example, suppose $F=\{f_1,f_2\}$ and $G=\{g_1,g_2\}$ meet
the root circle in the cyclic order $f_1,g_1,f_2,g_2$, with the
cut just before $f_1$. Then
\[
\bigl(\kappa(f_i,g_j)\bigr)=
\begin{pmatrix}0&0\\-1&0\end{pmatrix},
\]
so the correction is $-v_{f_2g_1}$. For the noninterlacing order
$f_1,f_2,g_1,g_2$, this pair contributes no correction to $z$.

\smallskip
\noindent\emph{Step 3: The cycle condition.}
To check $dz=0$, consider crossings $p<q<r$. The incoming cube
signs from $pq,pr,qr$ are $+,-,+$.
If all three crossings belong to one bundle, the resolution has
$s+1$ circles, so quantum grading $a+4$ requires one unmarked
$1$-label. Splitting the long interval circle gives
$\Delta(1)=1\otimes X+X\otimes1$: the unsigned image of $u_{pr}$
is the sum of those of $u_{pq}$ and $u_{qr}$. The signed sum vanishes.

If $x<x'$ belong to $F$ and $y$ belongs to $G>F$, the resolution
has $s-1$ circles. Its only state in quantum grading $a+4$ has
all unmarked labels $X$.
The merge from $u_{xx'}$ contributes $\varepsilon(x,x';y)$.
Each $v$-state contributes one before cube signs, by
$\Delta(X)=X\otimes X$ or the reduced marked split. The coefficient and its counterpart with
two crossings in $G$ are
\[
\varepsilon(x,x';y)-\kappa(x,y)+\kappa(x',y)=0,\qquad
\varepsilon(y,y';x)+\kappa(x,y)-\kappa(x,y')=0.
\]
These are \eqref{kh:eq:rowdiff} and \eqref{kh:eq:coldiff}.
Finally, three distinct bundles form a forest, whose resolution
has $s-3$ circles and minimum reduced quantum grading $a+6$.
There is no target in quantum grading $a+4$. This proves $dz=0$.

\smallskip
\noindent\emph{Step 4: Primitivity.}
Choose $x<y$ in a bundle incident to the marked circle; such a
pair exists because $D$ is reduced and $a>0$. Let $\lambda$
extract the coefficient of $u_{xy}$, and vanish on every other
basis state. Only the resolutions in which $x$ or $y$ is the sole
$1$-resolved crossing map to this resolution. Both incoming saddles are marked splits,
so they put $X$ on $B(x,y)$ and contribute zero to $\lambda$.
Thus $\lambda d^1=0$, whereas $\lambda(z)=1$.
These identities give the summand $\Z$ and remain valid over $R$.
\end{proof}

\begin{proof}[Proof of Theorem~\ref{kh:thm:fibered-all} for nontrivial prime links]
Choose a reduced connected positive diagram of $L$.
Lemma~\ref{kh:thm:cycle} gives a nonzero class over both $\Z$ and
$\F_2$. Together with the concentration and freeness of
Theorem~\ref{kh:thm:concentration} and the rank bounds of
Proposition~\ref{kh:prop:prime-bound}, this gives
$\rKh^2(L;R)\cong R\{a(L)+4\}$ for every marking and
$R=\Z,\F_2$.
\end{proof}

\subsection{Marking independence and connected sums}\label{kh:sec:sums}

A summing component need not contain the mark. We compare markings
in homological grading two, over both $\Z$ and $\F_2$. Over $\F_2$, independence holds for every link in all
homological gradings \cite[Section~3.2]{Shumakovitch14}.

\begin{lemma}\label{kh:lem:mark-independence}
For every fibered positive link $L$, every marked point
$b$, and $R=\Z$ or $\F_2$, the inclusion in the reduction sequence
induces an isomorphism
\begin{equation}\label{kh:eq:mark-independence}
\rKh^{2,a+4}(L,b;R)\xrightarrow{\ \cong\ }\Kh^{2,a+3}(L;R).
\end{equation}
Consequently the quantum-graded isomorphism type of $\rKh^2(L,b;R)$
is independent of the marked component.
\end{lemma}
\begin{proof}
In quantum grading $a+3$, the relevant part of
\eqref{kh:eq:reduction} is
\[
\rKh^{1,a+2}(L,b;R)\longrightarrow
\rKh^{2,a+4}(L,b;R)\longrightarrow
\Kh^{2,a+3}(L;R)\longrightarrow\rKh^{2,a+2}(L,b;R).
\]
The outer terms vanish by Lemma~\ref{kh:lem:initial} and
Theorem~\ref{kh:thm:concentration}. The isomorphism and quantum
concentration give independence of $b$.
\end{proof}

\begin{lemma}\label{kh:lem:connected-sum}
Let $L_1,L_2$ be fibered positive links, and let
$L=L_1\#L_2$ be their connected sum along specified components.
Put $a_i=a(L_i)$. For any marking $b$ of $L$, and any markings
$b_i$ of the factors,
\begin{equation}\label{kh:eq:degree-two-connected-sum}
\rKh^2(L,b;R)\cong
\rKh^2(L_1,b_1;R)\{a_2\}\oplus
\rKh^2(L_2,b_2;R)\{a_1\}.
\end{equation}
\end{lemma}
\begin{proof}
First take $b_i$ on the summing components and $b$ on the joined
component. The based connected sum formula
\cite[Lemma~2.15]{DunfieldLipshitzSchuetz26} gives a bigraded chain isomorphism
\[
\rC(L,b;R)\cong\rC(L_1,b_1;R)\otimes_R\rC(L_2,b_2;R).
\]
Both complexes are supported in nonnegative
homological gradings, with $H^0=R\{a_i\}$ and $H^1=0$.
The tensor terms in homological grading two are $H^2\otimes H^0$ and
$H^0\otimes H^2$. Over $\Z$, the possible K\"unneth torsion terms have indices
summing to three. Each vanishes because it has either a free factor
in homological grading zero or a zero factor in homological grading one.

The connected sum diagram is positive and has reduced Seifert
graph the wedge of the two trees at the joined Seifert circle.
It is fibered positive, with $a(L)=a_1+a_2$.
Lemma~\ref{kh:lem:mark-independence}, applied to $L$ and both factors,
now extends the formula in homological grading two to all the stated markings.
\end{proof}

\begin{proof}[Proof of Theorem~\ref{kh:thm:fibered-all} in general]
By Lemma~\ref{lem:diagram-decomposition}, a connected reduced
positive diagram $D$ decomposes into prime fibered positive
pieces representing $L_1,\ldots,L_m$, where $m=p(L)$.
Each $L_i$ contributes $R\{a(L_i)+4\}$ by the prime case proved above.
Repeated application of
Lemma~\ref{kh:lem:connected-sum} gives one copy of $R$ for each
nontrivial factor, all in quantum grading
$a(L_1)+\cdots+a(L_m)+4=a(L)+4$. Unknot pieces contribute zero
to homological grading two and have $a=0$. This proves the formula for every
marking and both coefficient rings.
\end{proof}

\begin{remark}\label{rmk:pbinduction}
Positive braids admit a simpler skein induction. At a square of an
Artin generator $\sigma^2$, the links in \eqref{eq:braid-square}
remain positive braid closures. Lemma~\ref{kh:lem:initial} and,
when $L_-$ splits, the disjoint-union formula \eqref{kh:eq:disjoint}
give $\rKh^1(L_-;R)=0$. The vanishing of the last term in
\eqref{kh:eq:main-les} lets the induction proceed without the
leaf-crossing calculations.
\end{remark}

\subsection{Concentration and the Jones polynomial}\label{kh:sec:concentration-jones}

Theorem~\ref{kh:thm:fibered-all} gives \emph{quantum concentration}:
homological grading two supports only quantum grading $a(L)+4$.
Fixing a quantum grading instead asks for \emph{homological
concentration}. The braid induction gives this at $a(L)+4$.

\begin{proposition}\label{kh:rem:braids}
For a nonsplit positive braid link $L$ and every marking, with
$a=a(L)$, we have
\begin{equation}\label{kh:eq:braid-quantum-row}
\rKh^{i,a+4}(L;R)\cong\Kh^{i,a+3}(L;R)\cong
\begin{cases}R^{p(L)},&i=2,\\0,&i\ne2.\end{cases}
\end{equation}
\end{proposition}
\begin{proof}
For a fibered positive link $J$, Manikandan's results
\cite[Lemma~3.2 and Proposition~1.9]{Manikandan26}, together with
the tree condition, give
$\rKh^{*,a(J)}(J;R)=R[0]$ and
$\rKh^{*,a(J)+2}(J;R)=0$.
Induct on $a$ using \eqref{eq:braid-square}, starting with the unknot.
The smoothing is nonsplit with $a(L_0)=a-1$. At quantum grading
$a+4$, the two columns of $E$ in \eqref{kh:eq:triangle} use quantum
gradings $a(L_0)+4$ and $a(L_0)+2$. By induction and the preceding
vanishing, $H^{*,a+4}(E)$ is supported only in homological grading two.

If $L_-$ is nonsplit, then $a(L_-)=a-2$, so its quantum grading
$a$ vanishes. Otherwise $L_-=L_1\sqcup L_2$ with
$a(L_1)+a(L_2)=a-1$. The disjoint-union formula
\eqref{kh:eq:disjoint} and the two lowest unreduced quantum
gradings \cite[Section~1.1]{KMMS} give
$\rKh^{*,a}(L_-;R)=R[0]$: with $L_1$ marked, only the factor
gradings $(a(L_1),a(L_2)+1)$ contribute. These groups are free,
so there are no integral K\"unneth torsion terms. After the shift
$[2]\{4\}$, this term is also supported only in homological
grading two. Exactness proves the reduced vanishing in
\eqref{kh:eq:braid-quantum-row}, and
Theorem~\ref{kh:thm:fibered-all} gives the rank. The unreduced
identification follows from \eqref{kh:eq:reduction}, since the
entire reduced quantum grading $a+2$ vanishes.
\end{proof}

For positive $3$-braids, this unreduced quantum grading appears
in the tables of del Valle V\'ilchez--Gonz\'alez-Meneses--Silvero
\cite[Theorem~1.1]{VMS25}. Equation~\eqref{kh:eq:braid-quantum-row}
allows any number of strands.

This distinction matters for the Jones polynomial. Each
coefficient is an alternating sum of dimensions over homological
gradings at a fixed quantum grading. For positive braid knots, the prime factor count appears
as a Jones coefficient \cite[Corollary~3]{Ito22}, in agreement with
\eqref{kh:eq:braid-quantum-row}. Quantum concentration alone does
not imply this equality for general fibered positive knots.

Let $D$ be a connected reduced positive diagram of a nontrivial
fibered knot $K$. Normalize the Jones polynomial so that $V_U(t)=1$ and positive
knots have positive minimal degree. In our Khovanov
conventions,
\[
\sum_{i,j}(-1)^i\dim\rKh^{i,j}(K;\F_2)\,q^j=V_K(q^2).
\]
Write
\[
V_K(t)=t^{g(K)}\bigl(1+c_2(K)t^2+\cdots\bigr).
\]
The missing linear term is Stoimenow's
\cite[Corollary~1]{Stoimenow05}. In the formula of Dasbach--Lin
\cite[Theorem~4.1]{DasbachLin06}, the tree condition gives
$e=v-1$ and no triangles. Every bundle has multiplicity at least
two, so $\mu=|V(\mc I_D)|$, while the number of mixed pairs is
$\theta=|E(\mc I_D)|$. Substituting these counts gives
\begin{equation}\label{eq:jones-interlacement}
c_2(K)=\mu-\theta
=b_0(\mc I_D)-b_1(\mc I_D)
=p(K)-b_1(\mc I_D),
\end{equation}
where the last equality uses
Lemma~\ref{prop:interlacement-count}.
The Jones coefficient differs from the reduced rank $p(K)$
of Theorem~\ref{kh:thm:fibered-all} by $b_1(\mc I_D)$. This numerical
comparison does not identify graph cycles with homology classes.

\begin{example}\label{kh:ex:10154}\label{kh:sec:difference}\label{rem:jones-comparison}
The knot $10_{154}$ is fibered and positive,
with genus three \cite[p.~256]{Stoimenow03}.
The reduced mod-two Khovanov table in KnotInfo \cite{Knotinfo} gives
\[
\rKh^{i,10}(10_{154};\F_2)=
\begin{cases}\F_2,&i=2,3,\\0,&\text{otherwise}.\end{cases}
\]
Although $\rKh^2(10_{154};R)$ has rank one by
Theorem~\ref{kh:thm:fibered-all}, the coefficient of $q^{10}$ in
$V_{10_{154}}(q^2)$ is $1-1=0$. Thus quantum concentration in homological
grading two does not imply homological concentration at $a+4$.
\end{example}

Even for positive braids, homological concentration can fail at the
next quantum grading. For the positive torus knot $T=T(3,4)$,
we have $a(T)=6$ and
$\rKh^{3,12}(T;\F_2)\cong\rKh^{4,12}(T;\F_2)\cong\F_2$,
as recorded in KnotInfo \cite{Knotinfo}.
Quantum grading $a+6$ already supports two homological
gradings. Connected sums of $T$ show that one quantum grading
can support arbitrarily many consecutive homological gradings
among positive braid knots.\footnote{For $K_m=\#^m T$, tensoring
these two classes over $\F_2$ gives nonzero groups in quantum
grading $12m$ at every homological grading from $3m$ to $4m$.
This uses composite knots and gives no such conclusion for prime knots.}

\section{Unreduced Khovanov homology and torsion}\label{kh:sec:unreduced}

We compute unreduced homology in homological grading two and use
the universal coefficient theorem to detect torsion in grading
three. We then prove the semi-adequate knot theorem.

\subsection{The unreduced calculation}\label{kh:subsec:unreduced-calculation}
The reduction sequence \eqref{kh:eq:reduction} leaves two quantum
gradings. The reduced calculation determines the lower group;
Lee homology and the following bound determine the upper group.

\begin{lemma}\label{kh:lem:diagonal-bound}
Let $D$ be a connected positive diagram with $c$ crossings and $s$
Seifert circles, and set $a=c-s+1$. Then
\[
\Kh^{i,j}(D;R)=0\quad\text{for }j-2i>a+1.
\]
In particular, $\Kh^{3,j}(D;R)=0$ for $j>a+7$.
\end{lemma}
\begin{proof}
In Wehrli's spanning tree model \cite[Theorem~2.2]{Wehrli04}, a
one-circle resolution with $r$ $1$-resolved crossings contributes
generators in diagonal gradings $-r\pm1$. The positive normalization
$[0]\{c\}$ changes these to $c-r\pm1$. Merging the $s$ circles
of the all-$0$ resolution requires $r\ge s-1$, giving
$j-2i\le c-s+2=a+1$. The integral chain homotopy equivalence
gives the same bound over $R$.
\end{proof}

\begin{proof}[Proof of Theorem~\ref{kh:thm:main-unreduced}]
Let $L$ be a fibered positive link, and write $a=a(L)$.
Since $\rKh^1(L,b;R)=0$, the reduction sequence gives
\begin{equation}\label{kh:eq:integral-reduction-final}
\begin{gathered}
0\longrightarrow\rKh^{2,j+1}(L,b;R)
\longrightarrow\Kh^{2,j}(L;R)
\longrightarrow\ker\delta\longrightarrow0,\\
\delta:\rKh^{2,j-1}(L,b;R)
\longrightarrow\rKh^{3,j+1}(L,b;R).
\end{gathered}
\end{equation}
$\Kh^2(L;R)$ is supported at $a+3$ and $a+5$.
We call the corresponding groups the lower and upper groups, respectively.
Lemma~\ref{kh:lem:mark-independence} identifies the lower group:
\begin{equation}\label{kh:eq:unred-lower}
\Kh^{2,a+3}(L;R)\cong\rKh^{2,a+4}(L,b;R)\cong R^{p(L)}.
\end{equation}
Over $\Z$, the upper group is free because it is a subgroup of the
free reduced group. In characteristic two, the splitting
\eqref{kh:eq:mod-two-splitting} already gives the complete answer:
\begin{equation}\label{kh:eq:unreduced-f2}
\Kh^2(L;\F_2)\cong
\F_2^{p(L)}\{a+3\}\oplus\F_2^{p(L)}\{a+5\}.
\end{equation}
It remains to compute the integral rank in quantum grading $a+5$.

The Lee differentials have bigrading $(1,4k)$ for $k\ge1$
(Subsection~\ref{kh:sec:lee}). From $(2,a+3)$ the first possible
target is $(3,a+7)$, which Lemma~\ref{kh:lem:diagonal-bound} allows.
From $(2,a+5)$ the targets have $i=3$ and $j\ge a+9$, so they
all vanish by the same lemma.
No differential enters the upper group since $\Kh^1(L;\Q)=0$.
As a result, it survives to $E_\infty$, giving
\[
\Kh^{2,a+5}(L;\Q)\cong
\operatorname{gr}_q^{a+5}H^2_{\mathrm{Lee}}(L).
\]

Since $a+3$ and $a+5$ are distinct modulo four, the upper group
is the only group in homological grading two in its quantum residue
class. Lemma~\ref{kh:lem:lee-residues} gives its dimension as $\ell(L)$.
Integral freeness gives the formula over $\Z$.
For a knot, $\ell(K)=0$, which gives the stated knot formula.
\end{proof}

\subsection{Torsion}\label{kh:sec:coefficients}\label{kh:sec:semiadequate}
The rank difference in Theorem~\ref{kh:thm:main-unreduced} detects
torsion in homological grading three; in the reduced theory the
ranks agree and exclude $2$-torsion there. Since the integral
complex is free, the universal coefficient sequence is
\begin{equation}\label{kh:eq:UCT}
0\longrightarrow\Kh^{i,j}(L;\Z)\otimes\F_2
\longrightarrow\Kh^{i,j}(L;\F_2)
\longrightarrow\Tor_1^{\Z}\bigl(\Kh^{i+1,j}(L;\Z),\F_2\bigr)
\longrightarrow0.
\end{equation}

\begin{theorem}\label{kh:thm:degree-three-torsion}
Let $L$ be a fibered positive link. Then
\begin{equation}\label{kh:eq:link-three-torsion}
T_2\bigl(\Kh^{3,j}(L;\Z)\bigr)\cong
\begin{cases}
(\Z/2)^{p(L)-\ell(L)},&j=a(L)+5,\\
0,&j\ne a(L)+5.
\end{cases}
\end{equation}
In particular, $\ell(L)\leq p(L)$. For a fibered positive knot,
\begin{equation}\label{kh:eq:knot-three-torsion}
T_2\bigl(\Kh^{3,a(K)+5}(K;\Z)\bigr)\cong(\Z/2)^{p(K)}.
\end{equation}
Every nontrivial fibered positive knot has $2$-torsion
in Khovanov homology in homological grading three.
\end{theorem}
\begin{proof}
Set $a=a(L)$. Since integral homology in homological grading two is free,
\eqref{kh:eq:UCT} with $i=2$ gives
\begin{equation}\label{kh:eq:rank-difference}
\dim T_2\bigl(\Kh^{3,j}(L;\Z)\bigr)
=\dim\Kh^{2,j}(L;\F_2)
-\rank\Kh^{2,j}(L;\Z).
\end{equation}
Substituting Theorem~\ref{kh:thm:main-unreduced} gives the link
formula. For a knot, $\ell(K)=0$, and $p(K)>0$ unless $K$ is the unknot.
\end{proof}

The exponent in Theorem~\ref{kh:thm:degree-three-torsion} has a
topological interpretation.

\begin{corollary}\label{kh:rem:hopf-unreduced}\label{kh:rem:torsion-exponents}
For a fibered positive link $L$, $\ell(L)$ is the number
of positive Hopf link prime factors. Thus $p(L)-\ell(L)$ counts
the remaining prime factors, and $\Kh^3(L;\Z)$ has no
$2$-torsion precisely when $L$ is an iterated connected sum of
positive Hopf links, including the unknot as the empty sum.
\end{corollary}
\begin{proof}
Form the connected graph on the components of $L$, with edges
weighted by their positive linking numbers. The partitions counted
by $\ell(L)$ correspond to its bridges of weight one. Connected sum identifies
vertices and preserves these bridges, so $\ell$ is additive.

By Lemma~\ref{lem:diagram-decomposition}, it remains to show that a
prime fibered positive link with $\ell>0$ is $H_+$. Color its
components by a partition of linking number one. A reduced positive
diagram has exactly two mixed crossings $x,y$. Each Seifert circle
changes color an even number of times, so $x,y$ join the same pair
$u,v$. Let $q$ be the bundle multiplicity. Smoothing the monochromatic
crossings outside this bundle leaves a $T(2,q)$ factor. Mixed crossings
require $q$ even, when all $q$ crossings are mixed, so $q=2$.

Every other Seifert circle is monochromatic. Each other bundle at
$u$ or $v$ therefore attaches within one monochromatic arc between $x,y$,
so the $uv$-bundle is isolated in $\mc I_D$. By primeness and
Lemma~\ref{lem:interlacement}, it is the only bundle, giving $H_+$.
Since $\ell(H_+)=1$, additivity proves the count. The torsion
statement follows from Theorem~\ref{kh:thm:degree-three-torsion}.
\end{proof}

The comparison counts even-order cyclic summands without
determining their orders: each $\Z/2^k$ has one element of order
two. The $2$-torsion subgroup need not be a direct summand or all
the torsion.

In the reduced theory, Theorem~\ref{kh:thm:fibered-all} gives the
same rank in homological grading two over $\Z$ and $\F_2$.
The universal coefficient sequence gives the following corollary.

\begin{corollary}\label{kh:cor:no-even-torsion}
For a fibered positive link with any marked point $b$,
\[
\rKh^2(L,b;\Z)\otimes\F_2
\xrightarrow{\ \cong\ }\rKh^2(L,b;\F_2),
\]
and $\rKh^3(L,b;\Z)$ has no $2$-torsion.
\end{corollary}

Manikandan \cite[Corollary~7.2]{Manikandan26} and
Kegel--Manikandan--Mousseau--Silvero \cite[Theorem~4.6]{KMMS}
use the same comparison one homological
grading lower. Here equality of ranks over $\Z$ and $\F_2$ is
essential: integral freeness alone does not exclude the torsion
contribution in \eqref{kh:eq:UCT}.

\begin{proof}[Proof of Theorem~\ref{kh:thm:positive-torsion-main}]
For a nonfibered positive knot, the reduced Seifert graph has positive
first Betti number \cite[Corollary~1]{Stoimenow05}, so
\cite[Theorem~1.1]{KMMS} gives $2$-torsion in homological grading two.
The fibered case follows from Theorem~\ref{kh:thm:degree-three-torsion}.
\end{proof}

We now extend the result to semi-adequate knots using the state graph
torsion results of Przytycki--Sazdanovi\'c. In the conventions of
Section~\ref{kh:sec:preliminaries}, the all-$0$ state graph has a
vertex for each state circle and an edge for each crossing. A
diagram is \emph{$0$-adequate} if this graph has no loops; define
$1$-adequacy similarly. A knot is semi-adequate if it has a diagram
adequate at either extreme.

\begin{proof}[Proof of Theorem~\ref{kh:thm:semiadequate-torsion-main}]
After mirroring $K$ if necessary, choose a connected $0$-adequate
diagram $D$ of $K$. Let $G$ be
its all-$0$ state multigraph, and let $G'$ replace every family
of parallel edges of $G$ by one edge. Both graphs are connected,
and $G$ has no loops.

If $G'$ contains an odd cycle, so does $G$, and
\cite[Equation~(1.1)]{PrzytyckiSazdanovic14} gives $2$-torsion in
$\Kh(K;\Z)$. If $G'$ is bipartite with a cycle,
\cite[Proposition~4.8(i)]{PrzytyckiSazdanovic14} gives the same conclusion.

Otherwise $G'$ is a tree. Its bipartition also colors $G$.
The adequate all-$0$ state surface is orientable, and
its boundary orientation makes $D$ positive
\cite[p.~591]{Kindred18}.
This state is the Seifert state, so $G'$ is
the reduced Seifert graph. Stoimenow's tree criterion
\cite[Proof of Corollary~1]{Stoimenow05} shows that $K$ is fibered.
Since $K$ is nontrivial, Theorem~\ref{kh:thm:degree-three-torsion}
gives $2$-torsion in $\Kh(K;\Z)$. Mirror duality gives the
same conclusion for the original knot.
\end{proof}

\section{A review of knot Floer homology}\label{sec:knot-floer}

We recall the knot Floer conventions and skein maps of
\cite[Section~3]{CH}, and then provide two lemmas for the induction.

\subsection{Filtered complexes and skein maps}

We use the knot Floer homology of Ozsv\'{a}th--Szab\'{o}
\cite{OSknot}, with the integer Maslov grading notation of \cite{CH}.
Alexander symmetry \cite[Section~8]{OSlink} gives
\begin{equation}\label{eq:alexander-symmetry}
\HFK_M(L,A)\cong\HFK_{M-2A}(L,-A).
\end{equation}
Let $C(L)$ denote $\HFK(L)$ equipped with a canceling differential
$\partial$ \cite[Section~3]{CH}. This differential strictly lowers
Alexander grading and lowers Maslov grading by one.
The homology of this chain complex gives
\begin{equation}\label{eq:total-homology}
H_*(C(L),\partial)\cong V^{\otimes(|L|-1)},
\qquad V=\F[0]\oplus\F[-1].
\end{equation}
The \emph{top term} is $\HFK(L,g(L))$; it has rank one when $L$
is fibered \cite{OScontact,HR22}.
Let $\tau_{\mathrm{top}}(L)$ be the least Alexander filtration level
containing a representative of the nonzero class in $H_0(C(L))$.
For a fibered positive link, strong quasipositivity gives
$\tau_{\mathrm{top}}(L)=g(L)$ \cite[Theorem~5.22]{HR22}, so
\begin{equation}\label{eq:fibered-positive-top}
\HFK(L,g(L))\cong\F[0].
\end{equation}
We choose the top generator to be a cycle by a filtered change of basis.

Let $(L_+,L_-,L_0)$ be an oriented skein triple. The filtered skein
maps \cite{OSknot,CH} take the form
\begin{equation}\label{eq:filtered-skein-maps}
C(L_+)\xrightarrow{f}C(L_-)\xrightarrow{g}Q
\xrightarrow{h}C(L_+),
\end{equation}
where
\begin{equation}\label{eq:floer-smoothing-complex}
Q=\begin{cases}
C(L_0),& |L_0|=|L_+|+1,\\
C(L_0)\otimes J,& |L_0|=|L_+|-1,
\end{cases}
\qquad
J=\F[0,1]\oplus\F^{\,2}[-1,0]\oplus\F[-2,-1].
\end{equation}
Here $\F[M,A]$ is supported in bigrading $(M,A)$, and $J$ has zero
differential.
The map $h$ preserves the Maslov grading, while $g$ lowers it
by one. All three maps preserve the Alexander filtration. Writing
$Q_M(A)$ for the summand of $Q$ in bigrading $(M,A)$, their components
preserving Alexander grading give the exact sequence
\begin{equation}\label{eq:floer-skein}
\begin{aligned}
\cdots\longrightarrow{}&\HFK_M(L_+,A)
\xrightarrow{f_A}\HFK_M(L_-,A)
\xrightarrow{g_A}Q_{M-1}(A)\\
\xrightarrow{h_A}{}&\HFK_{M-1}(L_+,A)\longrightarrow\cdots.
\end{aligned}
\end{equation}
The map $f_A$ has Maslov grading shift zero by Zemke's argument
\cite[Example~12.7]{ZemkeGradings19}.

On total homology, $G=g_*$ and $H=h_*$ fit into the exact sequence
\begin{equation}\label{eq:total-skein-sequence}
0\longrightarrow W\xrightarrow{G}W[-1]\oplus W[0]
\xrightarrow{H}W\longrightarrow0.
\end{equation}
Here $W=V^{\otimes(|L_+|-1)}$, $G$ includes the first summand with
Maslov shift $(-1)$, and $H$ projects onto the second.

Injectivity of $G$ leaves open whether $g_A$ is nonzero, since
a nonzero image class may admit a representative below filtration
$A$. The following criterion, adapted from
\cite[proof of Theorem~3.5]{CH}, excludes such representatives
in the induction.
The \emph{leading term} of a chain is its nonzero component of highest
Alexander grading.

\begin{lemma}\label{lem:filtered-observation}
Let $C$ be a filtered complex whose differential strictly lowers
Alexander grading and whose top group, at Alexander grading $A+1$,
is generated by a cycle in Maslov grading zero.
Suppose $H_{-1}(C)$ has a basis represented by cycles
$z_1,\ldots,z_d$ with linearly independent leading terms in
Alexander grading $A$. Then projection induces an injection
\[
\ell:H_{-1}(C)\longrightarrow C_{-1}(A),
\qquad \ell([z])=z|_A,
\]
with image spanned by these leading terms. Here $C_{-1}(A)$ denotes
the summand of $C$ in Maslov grading $(-1)$ and Alexander grading $A$.
In particular, no nonzero class in $H_{-1}(C)$ has a representative
below Alexander grading $A$.
\end{lemma}

\begin{proof}
There are no Maslov-$(-1)$ chains above Alexander grading $A$.
Boundaries have no component at $A$, since the only possible source
is the top group, generated by a cycle. Projection descends
to homology, and the independent leading terms prove the claims.
\end{proof}

\subsection{A graded lower bound}
Keeping the Maslov gradings in the Binns--Dey construction
\cite{BD23} gives the lower bound needed for the skein induction.

\begin{lemma}\label{prop:graded-bound}
If $L\subset S^3$ is a nontrivial fibered strongly quasipositive
link, then
\[
\dim\HFK_{-1}(L,g(L)-1)\ge |L|.
\]
\end{lemma}

\begin{proof}
Write $r=|L|$ and $g=g(L)$. The top generator has Maslov grading
zero \cite[Theorem~5.22]{HR22}, so Alexander symmetry
\eqref{eq:alexander-symmetry} places the bottom generator $c$ in
\[
(A,M)(c)=(-g,-2g).
\]
The supported contact structure is tight \cite[Proposition~5.24]{HR22},
so the monodromy is right-veering \cite[Theorem~1.1]{HKM07}.
Its inverse sends a nonseparating arc at each boundary component
to the left: the fixed-arc alternative in
\cite[proof of Theorem~4.1]{BD23} would give an $S^1\times S^2$
summand in $S^3$.

Apply their construction to the inverse open book. Since the construction
reverses the ambient orientation, the resulting diagram represents $L$.
The generators $d_1,\ldots,d_r$ have Alexander grading $1-g$ and
satisfy $\partial d_i=c$, so $M(d_i)=1-2g$.
Each bigon contains a single basepoint, on component $i$, so the componentwise
Alexander grading increases by one only there.
The $d_i$ give nonzero associated-graded classes
\cite[Lemma~4.4]{BD23}. These classes lie in distinct componentwise
Alexander gradings and remain independent in total Alexander grading
$1-g$. By Alexander symmetry \eqref{eq:alexander-symmetry}, we obtain
$r$ independent classes in Alexander grading $g-1$, all in Maslov grading
\[
(1-2g)-2(1-g)=-1. \qedhere
\]
\end{proof}

\section{Top term after the crossing change}\label{sec:crossing-change}

We prove the Maslov concentration of the top term using the generalized
Seifert surface of Feller--Lewark--Lobb \cite{FLL}. We also show
that it is connected and has index $g(L_+)-1$, locating the
Alexander grading needed for the skein sequence.

Using the tree and leaf hypotheses, we reduce the surface to the
standard Seifert surface of a pretzel link $P(-2,\ldots,-2)$ by
isotopies and removal of fibered positive Murasugi summands.
These operations preserve the Maslov-graded top
term, computed in Lemma~\ref{lem:pretzel-top}. Quasipositivity alone
does not determine this term. See Remark~\ref{rem:fll-comparison}
for a comparison with their reduction.

\begin{lemma}\label{lem:concentration}
Let $D$ be a connected reduced positive diagram of a prime fibered
positive link $L_+$, other than the positive Hopf link. There is a
crossing of $D$ incident to a leaf of its reduced Seifert graph whose
oriented smoothing is prime and for which the crossing-changed link
$L_-$ satisfies, with
$A=g(L_+)-1=g(L_-)$,
\[
\HFK(L_-,A)\cong\F^{\,b}[0]
\]
for some $b\ge1$.
\end{lemma}

\subsection{The generalized surface and its graded top term}
\label{subsec:generalized-surface}
\label{subsec:top-preserving}

Choose the crossing given by Lemma~\ref{lem:prime-smoothing}. Let $v$
be its leaf vertex, $u$ the adjacent vertex, and $q\ge2$ the number
of crossings between them. Change the chosen crossing to $c_-$ and
mark another $uv$-crossing as $c_+$. The diagram $D_-$ is of type II
in the terminology of \cite{FLL}. Let
\[
F_-=\Sigma'(D_-)
\]
be the generalized Seifert surface, with distinguished generalized
Seifert circles $s_1,s_2$. The proof of \cite[Proposition~2.1]{FLL}
shows that $F_-$ is quasipositive and hence has maximal Euler
characteristic. Seifert's algorithm applied to $D_-$ gives a surface
of Euler characteristic $\chi(F_-)-2$.
Throughout this section, crossing intervals in a generalized Seifert
diagram represent positive bands. The marked crossings $c_-,c_+$
are intersections of the projected boundaries of the distinguished
disks, and do not contribute additional bands.

Figure~\ref{fig:generalized-seifert-regions} shows the four regions
determined by $s_1,s_2$. The original Seifert circles $u,v$ are
the rounded boundaries of $O_1,O_2$, respectively.

\begin{figure}[htbp]
\centering
\includegraphics[width=.65\textwidth]{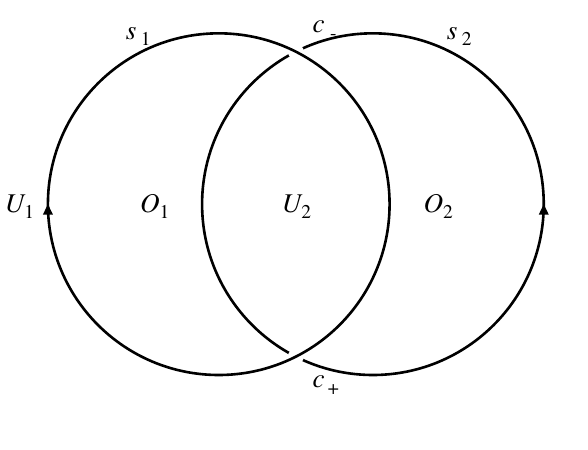}
\caption{The generalized circles and regions of \cite{FLL}.
The clockwise circle $s_1$ passes over the counterclockwise
circle $s_2$ at both marked crossings. The unbounded region is
$U_1$; oriented resolution recovers $u,v$ as the rounded
boundaries of $O_1,O_2$.}
\label{fig:generalized-seifert-regions}
\end{figure}

If $q\ge3$, an unmarked $uv$-crossing joins the distinguished
disks, and every other branch attaches to them, so $F_-$ is
connected. If $q=2$, primeness and Lemma~\ref{lem:interlacement}
give a bundle interlacing $uv$, unless $D$ is the excluded Hopf
diagram. Since $v$ is a leaf, this bundle is $uw$ for some
$w\ne v$. Its attachments lie on both arcs of $u$ cut by the
marked crossings, so the disk at $w$ connects the distinguished disks.

If $D$ has $c$ crossings and $s$ Seifert circles, its fiber
surface $F_+$ has Euler characteristic $s-c$. The generalized
surface has $s$ disks and $c-2$ bands, so
\[
\chi(F_-)=\chi(F_+)+2.
\]
For a Seifert surface $F$, put
$i(F)=(|\partial F|-\chi(F))/2$. Crossing change preserves the
number of components, and hence
\begin{equation}\label{eq:changed-surface-index}
i(F_-)=g(L_+)-1.
\end{equation}
Set $A=g(L_+)-1$. Maximal Euler characteristic and genus detection
give $g(L_-)=A$.

For a Seifert surface $F$ of maximal Euler characteristic, define its
\emph{top term} by
\[
\mc T(F)=\HFK(\partial F,i(F)).
\]
Gabai's theorem \cite{Gabai} and the graded Murasugi sum formula
of Cheng--Hedden--Sarkar \cite[Theorem 1.1]{CHS22} give the following
statement. We use \eqref{eq:alexander-symmetry} to express the formula at
the top Alexander grading. The Maslov shifts cancel because
$|\partial F|-1-2i(F)=\chi(F)-1$ and
$\chi(F)=\chi(F_1)+\chi(F_2)-1$.

\begin{proposition}\label{prop:top-murasugi}
Let $F=F_1*F_2$ be a Murasugi sum of Seifert surfaces. Then $F$
has maximal Euler characteristic among Seifert surfaces of its
boundary if and only if $F_1$ and $F_2$ each have maximal Euler
characteristic for their boundaries. In this case, the top term
$\mc T(F)$ satisfies
\begin{equation}\label{eq:top-murasugi}
\mc T(F)\cong\mc T(F_1)\otimes\mc T(F_2)
\end{equation}
as absolutely Maslov-graded vector spaces.
\end{proposition}

A fibered positive surface, including a positive Hopf band, has
top term $\F[0]$ by \eqref{eq:fibered-positive-top}. Removing such
a Murasugi summand therefore preserves $\mc T$ and maximal Euler
characteristic. Isotopy preserves both as well.

We use the following three operations:
\begin{enumerate}[label=(\roman*)]
\item Surface isotopies, including the crossing swaps in
Figure~\ref{fig:generalized-seifert-moves}(a), the disk--crossing
exchanges in Figure~\ref{fig:generalized-seifert-moves}(b), and
cancellation of a terminal disk with its
single incident band.
\item Murasugi splitting along an ordinary Seifert circle, as in
step~$(4')$ of \cite{FLL}. Retain the summand containing the marked
pair and remove positive factors whose reduced Seifert graphs are
trees; these factors are fibered by \cite{Stoimenow05}.
\item Deletion of one of two positive crossing intervals that are
\emph{next to each other} in the sense of
\cite[Definition 2.3]{FLL}, using \cite[Lemma 2.7]{FLL}.
\end{enumerate}

Adjacency in \cite[Figure~2]{CH} requires consecutive crossings
on both Seifert circles. For operation~(iii), two intervals are
next to each other on a circle $k$ if they lie on the same side
of $k$ and an arc of $k$ between their endpoints avoids $c_-,c_+$
and all other attachments on that side. Attachments on the opposite
side are allowed. For deletion, the intervals must join the same
two circles and satisfy this condition on both. The two arcs and
the intervals must bound a region containing no generalized circles
or crossing intervals. These conditions may hold after swapping
one interval over $c_-$ or $c_+$.
\begin{figure}[htbp]
\centering
\includegraphics[width=.70\textwidth]{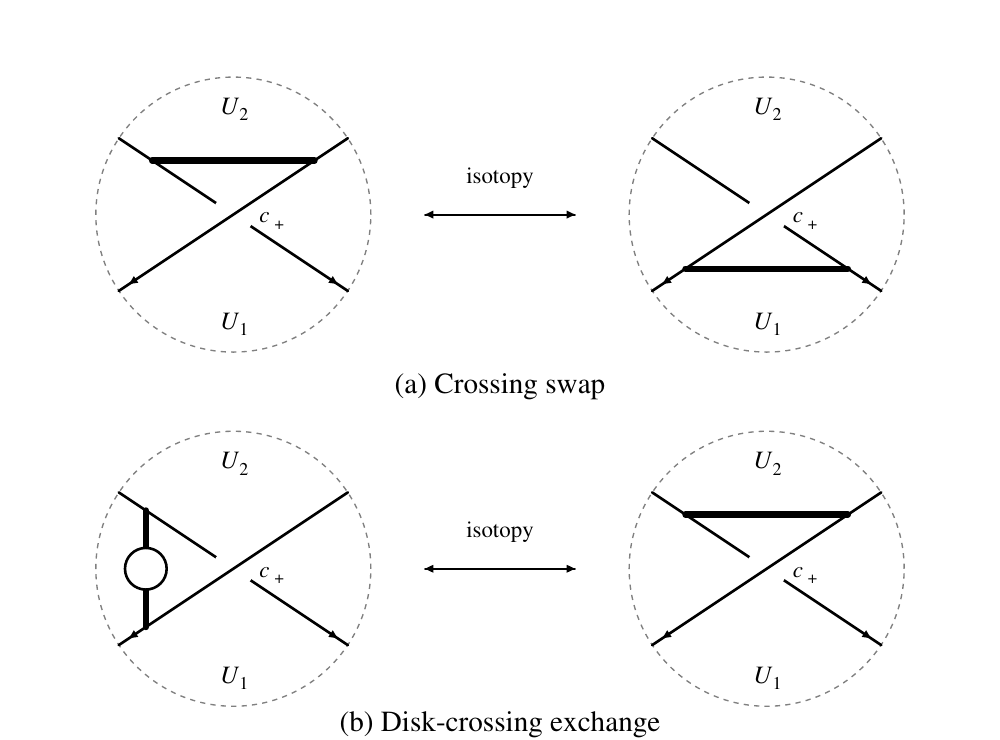}
\caption{Reversible local isotopies adapted from \cite{FLL}.
In (a) we swap a direct band across $c_+$. In (b) we exchange an
ordinary disk and its two bands for a direct $U_2$-band. Thin curves
bound disks and thick intervals are positive bands. Each move is
supported in the indicated disk, which contains only the part of
the diagram shown.}
\label{fig:generalized-seifert-moves}
\end{figure}

For the third operation, write $F=\Sigma'(D)$ and
$F'=\Sigma'(D')$ before and after the deletion chosen in
\cite[Lemma 2.7]{FLL}. Its proof gives, up to isotopy,
\[
F=S*(B'*\mathcal H_+),\qquad F'=S*B',
\]
where $\mathcal H_+$ is a positive Hopf band. Here $S$ is a disk in the nonnested case and the common
complementary Murasugi summand in the nested case.
By Proposition~\ref{prop:top-murasugi}, the deletion preserves
maximal Euler characteristic and gives $\mc T(F)\cong\mc T(F')$.
We do not need $S$ to be fibered.

\subsection{Reduction of the ordinary branches}

Consecutive bands at a terminal disk can enclose branches attached
to its parent. The next lemma lets us clear the region between the
bands relative to its boundary. The Hopf deletion of
\cite[Lemma~2.7]{FLL} then applies, even if attachments remain on
the other side of the boundary arc.

\begin{lemma}\label{lem:relative-pruning}
In a generalized Seifert diagram, let $R$ be a closed disk region
containing no marked crossing,
with a fixed boundary arc $\gamma$ on an ordinary or distinguished
Seifert circle. Suppose that its ordinary disks have empty
interiors, all bands are positive, and the reduced incidence graph
of these disks together with $\gamma$ is a tree. Suppose also that
all attachments to the exterior lie on $\gamma$. Then the disks
and bands in $R$ can be removed relative to its boundary by
terminal disk--band cancellations and the Hopf deletions of
\cite[Lemma 2.7]{FLL}. These operations preserve $\mc T$ of the
full surface.
\end{lemma}

\begin{proof}
Induct on the number of disks and bands in $R$, excluding the
fixed boundary circle. Root the incidence tree at $\gamma$ and
choose a terminal disk $k$, with parent $p$. Cancel $k$ if it
has one band. Otherwise, the sectors between consecutive
$k$--$p$ bands are disjoint. If $p$ is an interior disk, choose
a sector away from the fixed boundary; at most one sector can
contain that boundary. If $p$ is the boundary circle, choose
consecutive bands whose intervening $p$-subarc lies in $\gamma$.
In either case they cut off a disk region $R'$ whose boundary
consists of the two bands and the intervening subarcs of
$\partial k$ and $\partial p$.

The subarc on $k$ has no other attachments. Empty disk interiors
and disjoint bands force all enclosed material to attach through
the retained parent arc on $p$, with a tree as reduced incidence
graph. Induction empties $R'$ relative to this arc and the bounding
bands; the smaller problem excludes $k$ and those two bands.
The bands are now next to each other
\cite[Definition 2.3]{FLL}, so \cite[Lemma 2.7]{FLL} deletes one.
This decreases the band count while preserving connectedness.
Induction completes the removal using the operations of
Subsection~\ref{subsec:top-preserving}.
\end{proof}

Now we apply these operations in the four regions of
Figure~\ref{fig:generalized-seifert-regions}.

\begin{lemma}\label{lem:bridge-form}
The generalized surface reduces by the operations of
Subsection~\ref{subsec:top-preserving} to a diagram with the following
properties:
\begin{enumerate}[label=(\roman*)]
\item $O_2$ contains no ordinary circles or crossing intervals.
\item $U_1,U_2$ contain no ordinary circles and each contains at
most one direct band between $s_1,s_2$.
\item $O_1$ contains finitely many pairwise non-nested ordinary
circles, each joined by exactly one band to each distinguished
circle. No bands join different ordinary circles.
\end{enumerate}
The retained surface is connected and has the same absolutely
Maslov-graded top term $\mc T$.
\end{lemma}

\begin{proof}
Since $v$ is a leaf, no ordinary branch attaches to the
$v$-boundary of $O_2$. A band in $O_2$ would have both endpoints on $v$, which is
impossible in a Seifert graph. Connectedness forces
$O_2$ to be empty.

Work on the projection sphere, so that every region has disk
closure, and perform the reductions relative to its boundary.
Deleting $u$ from the original reduced Seifert tree leaves
components with unique roots adjacent to $u$, and no bands join
distinct components. In a $U_i$-region, each ordinary component
attaches only to the boundary portion belonging to $u$, since
$v$ has no other neighbor. In $O_1$, a root can meet both portions.

\smallskip
\noindent\emph{Step 1: Remove nesting.}
Split along ordinary circles enclosing material away from the
region boundary, retaining the marked pair. Each removed positive
surface has a subtree as reduced Seifert graph, so it is fibered
\cite{Stoimenow05} with top term $\F[0]$.
Proposition~\ref{prop:top-murasugi} applies. Repeat until all
ordinary disk interiors are empty.

\smallskip
\noindent\emph{Step 2: Prune descendants.}
Retain the roots and their boundary attachments. Cancel a
terminal nonroot disk $k$ with one band. Otherwise, successive
bands to its parent $p$ bound a region $R$ away from the fixed
boundary. The enclosed material attaches through the parent arc
and has a subtree as reduced incidence graph.
Lemma~\ref{lem:relative-pruning} empties $R$ relative to that arc
and the bounding bands; see Figure~\ref{fig:descendant-region}.
Delete a band by \cite[Lemma 2.7]{FLL}, repeat until $k$ cancels,
and continue through the descendants.

\begin{figure}[htbp]
\centering
\includegraphics[width=.70\textwidth]{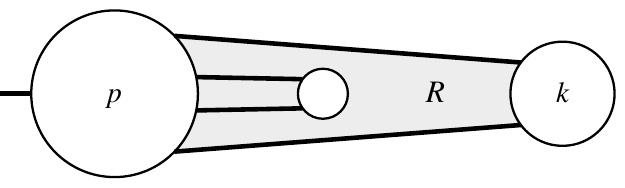}
\caption{A branch in $R$ attached to the parent $p$ of the terminal
disk $k$. All circle interiors are empty.}
\label{fig:descendant-region}
\end{figure}

\smallskip
\noindent\emph{Step 3: Remove one-sided roots.}
Successive bands from a one-sided root cut off a cap avoiding
the other boundary portion and the marked crossings. Its branches
attach to the fixed boundary arc, so Lemma~\ref{lem:relative-pruning}
clears the cap. A root meeting both portions cannot lie inside,
since a band would cross the cap boundary. Delete a band by
\cite[Lemma 2.7]{FLL} and repeat until the root cancels. This
removes all one-sided roots, including those in $U_1,U_2$;
Figure~\ref{fig:cap-pruning} shows an occupied cap.

\begin{figure}[htbp]
\centering
\includegraphics[width=\textwidth]{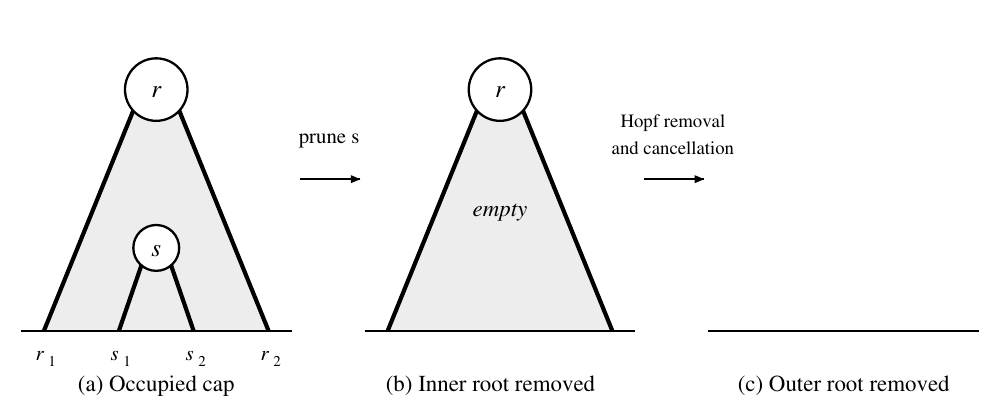}
\caption{Cap pruning: remove the inner root $s$ to obtain (b),
then the outer root $r$ to obtain (c), each by Hopf removal and
disk--band cancellation. Thin curves bound disks; thick intervals
are positive bands with twists suppressed. Subscripts label bands.
The shaded cap is complementary to the surface.}
\label{fig:cap-pruning}
\end{figure}

\smallskip
\noindent\emph{Step 4: Reduce the remaining bands.}
Each remaining root in $O_1$ meets both boundary portions.
Its bands to either portion are consecutive, since disjoint bands
preserve cyclic order. Successive bands bound an empty cap:
one-sided material has been removed, and no root meeting both
portions can lie inside. Hopf deletions leave one band to each
portion. No bands join distinct roots.

The regions $U_1,U_2$ now contain only direct bands. Successive
bands bound empty strips, so Hopf deletions leave at most one in
each region. Attachments on the opposite side of a boundary arc
are allowed by \cite[Definition 2.3]{FLL}.

Each deletion or cancellation decreases the number of disks and
bands. All operations preserve $\mc T$. They also preserve connectedness:
a parallel-band deletion retains the other band, disk--band
cancellation is an isotopy, and Murasugi splitting retains the
connected summand containing the marked pair.
\end{proof}

We call the configuration of Lemma~\ref{lem:bridge-form}
\emph{bridge form}. Figure~\ref{fig:pretzel-absorption}(a) shows
three bridges with $U_1$ and $U_2$ empty. Shrinking each ordinary
disk and its bands to an arc across $O_1$ shows that the bridges
have the same order on both boundary portions from $c_-$ to $c_+$.
Direct $U_i$-band endpoints may interleave with them; we retain this order.

\subsection{Reduction to pretzel surfaces and their top terms}

\begin{lemma}\label{lem:pretzel-reduction}
Every connected surface in bridge form reduces, preserving
$\mc T$, to the standard Seifert surface $F_m$ of the pretzel link
\begin{equation}\label{eq:pretzel-family}
P_m=P(\underbrace{-2,\ldots,-2}_{m\text{ entries}}),\qquad m\ge1,
\end{equation}
with the orientation induced by the surface and the sign
convention for pretzel links of \cite{FLL}. The surface $F_1$ is a disk.
\end{lemma}

\begin{proof}
Let $n$ be the number of bridge disks and $e\in\{0,1,2\}$ the
number of direct bands. If $e=0$, connectedness gives $n\ge1$.
Absorb each ordinary disk into its two bands to obtain one
$(-2)$-strip. Their common order along the boundary portions
gives the surface $F_n$. This is the
local identification used for $D_5$ in \cite[Lemma 2.6]{FLL},
with the same framing for each bridge.
Figure~\ref{fig:pretzel-absorption} shows $n=3$.

\begin{figure}[htbp]
\centering
\includegraphics[width=\textwidth]{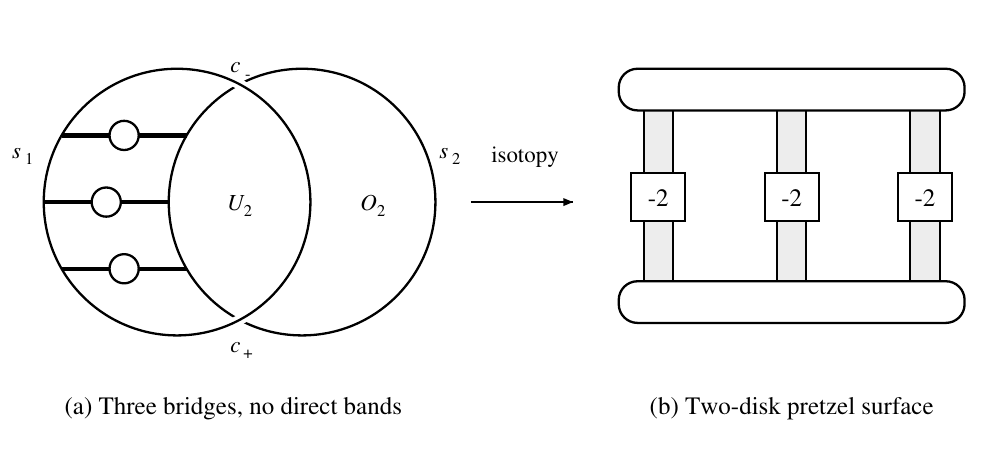}
\caption{Absorbing three bridges when $e=0$: the diagram in (a)
becomes the pretzel surface $F_3$ in (b), with boundary
$P(-2,-2,-2)$. Each $(-2)$-box denotes two signed half twists
in the convention of \cite{FLL}.}
\label{fig:pretzel-absorption}
\end{figure}

Suppose $e>0$. Let $\alpha,\beta$ be the boundary arcs of $O_1$
adjacent to $U_1,U_2$, ordered from $c_-$ to $c_+$. The bridges
have the same order on both. Each direct band has one endpoint
on the corresponding arc and one bordering $O_2$. The latter has
an unobstructed route to $c_+$, since $O_2$ is empty and each
$U_i$ has at most one band.

Call a band or bridge \emph{accessible} if the boundary intervals
for its move across $c_+$ contain no intervening attachments.
A direct band ending after the last bridge on its $O_1$-boundary
arc is accessible, since its interval bordering $O_2$ is clear.
Otherwise the last bridge is accessible on both arcs. For $n=0$,
a direct band is immediately accessible; see Figure~\ref{fig:bridge-order}.

\begin{figure}[htbp]
\centering
\includegraphics[width=\textwidth]{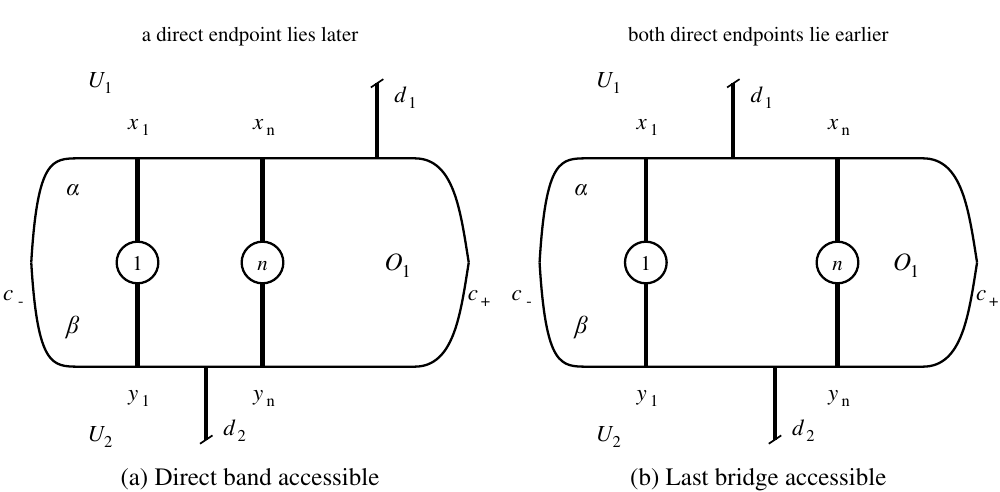}
\caption{Accessibility in bridge form. The boundary arcs
$\alpha,\beta$ of $O_1$ are ordered from $c_-$ to $c_+$.
Only the $O_1$-boundary endpoints of $d_1,d_2$ are shown; their
other endpoints border the empty region $O_2$.
In (a), $d_1$ is accessible; in (b), the last bridge is accessible.}
\label{fig:bridge-order}
\end{figure}

If a bridge is accessible, exchange it for a direct $U_2$-band
using Figure~\ref{fig:generalized-seifert-moves}(b). If necessary,
swap this new band across $c_+$ to the region containing an
existing direct band. After the swap, any intervening bridge
attachments lie on the $O_1$ side of the boundary arc, opposite
the prospective Hopf strip in $U_i$. The strip is empty in the sense of
\cite[Definition~2.3]{FLL}, which allows these attachments on
the opposite side. We can now apply \cite[Lemma~2.7]{FLL} to
delete a Hopf band. This reduces $n+e$ by one and
restores at most one direct band in each $U_i$.

If a direct band is accessible and $e=2$, swap it into the other
band's region and apply the same deletion, reducing $e$ by one.
If $e=1$, first move it to $U_2$ if necessary and reverse the
exchange in Figure~\ref{fig:generalized-seifert-moves}(b).
This terminal isotopy creates one bridge and gives $e=0$, the
case already treated. When $n=0$ and $e=1$, it gives the disk $F_1$.
The local moves are illustrated in
Figure~\ref{fig:reduction-local-moves}.

Accessibility supplies the empty disks for the isotopies, and
each nonterminal step decreases $n+e$. The moves preserve $\mc T$
by Subsection~\ref{subsec:top-preserving} and preserve connectedness,
since a Hopf deletion retains a parallel band. Consequently, the
procedure ends at $F_m$ with $m\ge1$.
\end{proof}

\begin{figure}[!htbp]
\centering
\includegraphics[width=\textwidth]{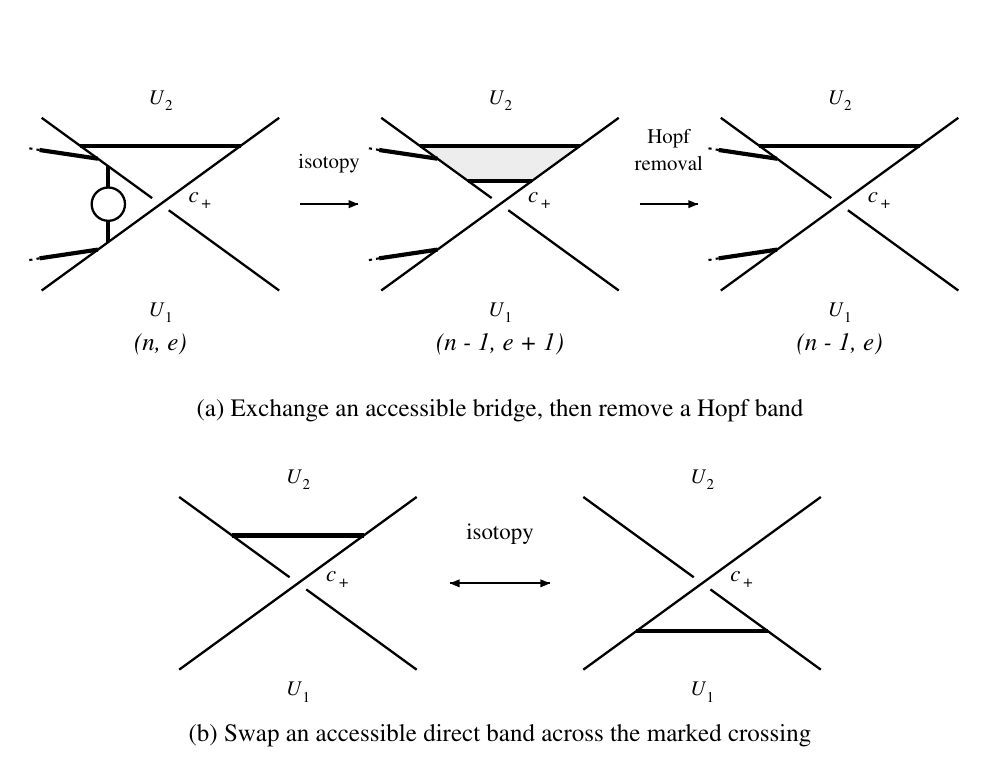}
\caption{Local reductions at $c_+$. In (a) we exchange an accessible
bridge for a direct band and then remove a Hopf band across the empty
shaded strip. In (b) we swap a direct band between $U_1$ and $U_2$.
Thin arcs bound distinguished disks and thick intervals are positive
bands. The surviving band is drawn schematically. We choose the
deleted band as in the proof of \cite[Lemma~2.7]{FLL}. For $e=1$,
reversing the first move in (a)
with no other direct band gives $e=0$.}
\label{fig:reduction-local-moves}
\end{figure}

\begin{remark}\label{rem:fll-comparison}
The tree gives each ordinary branch a unique root and makes the
removed positive pieces fibered. The leaf makes $O_2$ empty,
allowing the accessibility argument for every retained attachment order.

The reduction to the diagrams in \cite[Figure~7]{FLL} also uses
step~$(6')$, which adds a $1$-handle. The original surface inherits
quasipositivity as an incompressible subsurface of the enlarged
surface. The graded Murasugi formula does not show that this
enlargement preserves $\mc T$. Under the tree and leaf hypotheses,
the operations of Subsection~\ref{subsec:top-preserving} suffice to
reach $F_m$, and each preserves $\mc T$.
\end{remark}

For the residual pretzel surfaces, we have the following.

\begin{lemma}\label{lem:pretzel-top}
For the oriented surfaces in \eqref{eq:pretzel-family},
\begin{equation}\label{eq:pretzel-top}
i(F_m)=m-1,\qquad
\mc T(F_m)=\HFK(P_m,m-1)\cong\F^{\,m}[0].
\end{equation}
\end{lemma}

\begin{proof}
The surface $F_m$ has genus zero and $m$ boundary components, so
$i(F_m)=m-1$ and $g(P_m)\le m-1$. Induct on $m$ to compute
Alexander grading $m-1$, starting with the unknot at $m=1$.

For $m\ge2$, choose a crossing in an outermost pretzel strip.
With the boundary orientation of $F_m$, the strands in each
$(-2)$-strip are oppositely oriented, so its crossings are positive
despite the negative twist notation. The oriented skein triple is
\[
L_+=P_m,\qquad L_0=P_{m-1},\qquad
L_-=C_m:=P(0,\underbrace{-2,\ldots,-2}_{m-1\text{ entries}}).
\]
The smoothing deletes the strip after a nugatory crossing is
removed. The crossing change leaves an untwisted strip; absorbing
it into the disks gives a boundary connected sum of $m-1$ positive
Hopf annuli with their boundary orientations.
Figure~\ref{fig:pretzel-skein} shows the case $m=3$.

\begin{figure}[htbp]
\centering
\includegraphics[width=\textwidth]{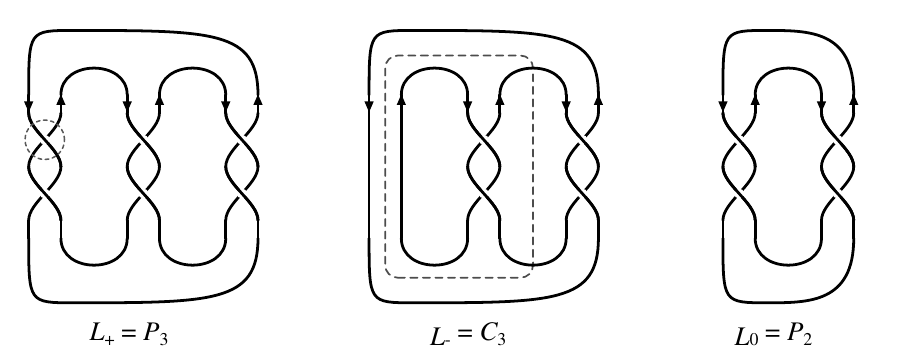}
\caption{The oriented skein triple for $P_3$, at the circled
positive crossing. In the middle diagram, the opposite crossings
have been canceled; the dashed curve exhibits $C_3$ as a connected
sum of two positive Hopf links. In the right diagram, the nugatory
crossing remaining after the oriented smoothing has been removed.}
\label{fig:pretzel-skein}
\end{figure}

Since the positive Hopf link has top term $\F[0]$ in Alexander
grading one, the connected sum formula
\cite[Product Formulae~3.3]{CH} gives
\[
g(C_m)=m-1,\qquad \HFK(C_m,m-1)\cong\F[0].
\]
Set $A=m-1$. Since $|L_0|=|L_+|-1$, the smoothing term in
\eqref{eq:floer-skein} is $Q=C(P_{m-1})\otimes J$. By induction,
$g(P_{m-1})=m-2$, so only the summand $\F[0,1]$ of $J$
contributes at Alexander grading $A$. Hence
\[
Q_*(A)\cong\HFK(P_{m-1},m-2)\cong\F^{\,m-1}[0],
\]
with no Maslov shift. The relevant part of \eqref{eq:floer-skein} is
\[
\cdots\longrightarrow Q_M(A)\longrightarrow\HFK_M(P_m,A)
\longrightarrow\HFK_M(C_m,A)\longrightarrow Q_{M-1}(A)
\longrightarrow\cdots.
\]
The two known terms are supported in Maslov grading zero.
Exactness gives vanishing of $\HFK_M(P_m,A)$ for $M\ne0$ and
a short exact sequence
\begin{equation}\label{eq:pretzel-skein-short}
0\longrightarrow\F^{\,m-1}[0]\longrightarrow\HFK(P_m,m-1)
\longrightarrow\F[0]\longrightarrow0.
\end{equation}
Thus $\HFK(P_m,m-1)\cong\F^{\,m}[0]$. Its nonvanishing,
together with the surface upper bound, gives $g(P_m)=m-1$
and completes the induction.
\end{proof}

In particular, the annular surface for $P(-2,-2)$ has oriented
boundary the mirror of $\mathrm{L4a1}\{0\}$ listed in LinkInfo
\cite{Knotinfo}. This oriented link has top term $\F^{\,2}[0]$.

\subsection{Concentration of the top term and a surviving cycle}
\label{subsec:surviving-class}

\begin{proof}[Proof of Lemma~\ref{lem:concentration}]
The surface $F_-$ of Subsection~\ref{subsec:generalized-surface} is
connected, quasipositive, and has index $A=g(L_+)-1$.
Lemmas~\ref{lem:bridge-form} and~\ref{lem:pretzel-reduction} reduce
it to $F_m$ for some $m\ge1$, preserving its absolutely
Maslov-graded top term. Lemma~\ref{lem:pretzel-top} gives
\[
\HFK(L_-,A)=\mc T(F_-)
\cong\mc T(F_m)\cong\F^{\,m}[0].
\]
Equation~\eqref{eq:changed-surface-index} and genus detection
identify $A$ as the highest Alexander grading, proving the claim
with $b=m$.
\end{proof}

For the induction, we also need a surviving cycle. The link
$L_-$ is strongly quasipositive by \cite{FLL}, so
$\tau_{\mathrm{top}}(L_-)=g(L_-)=A$ by
\cite[Proposition 5.19]{HR22}, with the link convention of
\cite{CH}. We can choose a cycle
\[
\xi\in C(L_-),\qquad M(\xi)=0,
\]
of filtration $A$ representing the nonzero class in $H_0(C(L_-))$.
The cycle may have correction terms in lower filtration.

\begin{remark}\label{rem:split-exception}
Concentration in Maslov grading zero fails when $L_-=J_1\sqcup J_2$
with both $J_i$ fibered and positive, allowing unknots. The
split union formula gives
\[
\HFK(L_-,g(L_-))\cong\F[0]\oplus\F[-1].
\]

This occurs in the noninterlacing multiplicity-two case of
Lemma~\ref{kh:cor:split-groups}. For example, changing a crossing
in one summand of $H_+\#H_+$ gives $H_+\sqcup U$.
\end{remark}

\section{The knot Floer calculation}\label{sec:theorem-proofs}

The preferred crossing gives a smaller prime smoothing and the
concentration of Lemma~\ref{lem:concentration}. It remains to show
that the skein map $g_A$ has rank one. The graded lower bound on
knot Floer homology gives $\operatorname{rank}g_A\le1$; surviving
cycles carried through the induction give $\operatorname{rank}g_A\ge1$.

\begin{proof}[Proof of Theorem~\ref{thm:links}]
Along with the formula, we prove that $H_{-1}(C(L))$ has a basis
of $|L|-1$ cycles with independent leading terms in Alexander
grading $g(L)-1$. This allows us to apply
Lemma~\ref{lem:filtered-observation} to the smoothing complex.

\smallskip
\noindent\emph{Step 1: Set up the induction.}
Induct on the crossing number of a connected positive diagram
$D$, removing nugatory crossings as needed.

For the unknot, the next-to-top term is zero and the cycle assertion
is empty. For the positive Hopf link,
\[
\HFK(H_+,1)=\F[0],\qquad
\HFK(H_+,0)=\F^{\,2}[-1].
\]
Its total homology in Maslov grading $(-1)$ is one-dimensional and has a cycle
representative with nonzero leading term in Alexander grading zero.
Both assertions hold in the base cases.

If $D$ is not prime, Lemma~\ref{lem:diagram-decomposition} gives
$D=D_1\#D_2$ with fewer crossings in each nontrivial factor.
The links $L_1,L_2$ are fibered positive, with surviving top cycles
in Maslov grading zero. The connected sum formula \cite{OSknot,OSlink} gives
\begin{equation}\label{eq:sum-next-top}
\begin{split}
\HFK(L,g(L)-1)\cong{}&
\HFK(L_1,g(L_1)-1)\otimes\HFK(L_2,g(L_2))\\
&\oplus\HFK(L_1,g(L_1))\otimes\HFK(L_2,g(L_2)-1).
\end{split}
\end{equation}
Since $|L|=|L_1|+|L_2|-1$ and $p(L)=p(L_1)+p(L_2)$,
induction gives the required formula. Tensoring the inductive
cycles with the other factor's top cycle gives $|L|-1$ independent
total homology classes, with independent leading terms in the two
summands of \eqref{eq:sum-next-top}.

For a prime link $L_+$ other than the base cases, choose the
crossing of Lemmas~\ref{lem:prime-smoothing} and~\ref{lem:concentration}.
Put $r=|L_+|$ and $A=g(L_+)-1$. The smoothing $L_0$ is nontrivial,
prime, fibered positive, and has fewer crossings. Removing one
band from the fiber surface gives
\begin{equation}\label{eq:smoothing-genus}
g(L_0)=\begin{cases}
A+1,& |L_0|=r+1,\\
A,& |L_0|=r-1.
\end{cases}
\end{equation}
Let $Q$ be the smoothing complex defined in
Section~\ref{sec:knot-floer}. In both cases its highest Alexander
grading is $A+1$, with a one-dimensional group generated by a
cycle in Maslov grading zero. The two component cases are
summarized by
\[
\begin{array}{c|c|c|c}
|L_0|&g(L_0)&\dim Q_{-1}(A)&\dim H_{-1}(Q)\\ \hline
r+1&A+1&r+1&r\\
r-1&A&(r-1)+2=r+1&(r-2)+2=r.
\end{array}
\]
The inductive formula for $L_0$, with the shifts in
\eqref{eq:floer-smoothing-complex}, gives
\begin{equation}\label{eq:smoothing-next-top}
Q_{-1}(A)\cong\F^{\,r+1},\qquad
Q_M(A)=0\quad\text{for }M\notin\{0,-1\}.
\end{equation}

The first row follows from induction, since $p(L_0)=1$.
For the second row,
\[
Q_{-1}(A)=
\bigl(\HFK_{-1}(L_0,A-1)\otimes J_{(0,1)}\bigr)
\oplus\bigl(\HFK_0(L_0,A)\otimes J_{(-1,0)}\bigr),
\]
where $J_{(M,A)}$ denotes the summand of $J$ in bigrading $(M,A)$.
Induction provides $r-2$ surviving cycles from the first summand.
Tensoring the top cycle with the two generators of $J_{(-1,0)}$
gives two more from the second. In both cases the $r$ cycles
represent a basis of $H_{-1}(Q)$ by \eqref{eq:total-homology},
and their leading terms in Alexander grading $A$ are independent.
Lemma~\ref{lem:filtered-observation} applies.

\smallskip
\noindent\emph{Step 2: Determine the rank of the skein map.}
Concentration for $L_-$ gives the exact sequence
\begin{equation}\label{eq:rank-one-sequence}
\HFK_0(L_-,A)\xrightarrow{g_A}Q_{-1}(A)
\xrightarrow{h_A}\HFK_{-1}(L_+,A)\longrightarrow0.
\end{equation}
Taking dimensions gives
\begin{equation}\label{eq:rank-one-dimension}
r+1-\operatorname{rank}g_A
=\dim\HFK_{-1}(L_+,A)\ge r,
\end{equation}
where the inequality follows from Lemma~\ref{prop:graded-bound}.

Take the cycle $\xi$ of
Subsection~\ref{subsec:surviving-class}. Injectivity of $G=g_*$ gives
$G[\xi]\ne0$ in $H_{-1}(Q)$. Applying the injection $\ell$ of
Lemma~\ref{lem:filtered-observation} gives
\[
0\ne\ell(G[\xi])=g(\xi)|_A=g_A(\xi|_A).
\]
The last equality holds because $\xi$ has filtration $A$ and $g$
is filtered. Together with \eqref{eq:rank-one-dimension}, this gives
\begin{equation}\label{eq:skein-map-rank}
\operatorname{rank}g_A=1,\qquad
\HFK_{-1}(L_+,A)\cong\F^{\,r}.
\end{equation}
For $M\notin\{0,-1\}$, the adjacent $L_-$ groups vanish by
Lemma~\ref{lem:concentration}, and $Q_M(A)=0$ by
\eqref{eq:smoothing-next-top}. Exactness gives
$\HFK_M(L_+,A)=0$. Denoting the dimension in Maslov grading zero by $z(L_+)$
proves the formula in Theorem~\ref{thm:links} for a prime link.

\smallskip
\noindent\emph{Step 3: Prove the cycle assertion.}
Choose cycles $z_1,\ldots,z_r$ in $Q$ representing a basis of
$H_{-1}(Q)$ with independent leading terms in Alexander grading
$A$. Their leading terms span $E=\ell(H_{-1}(Q))\subset Q_{-1}(A)$.
By Step~2, $\ell(G[\xi])$ spans the image of $g_A$. Since $[\xi]$
generates $H_0(C(L_-))$, we have
\[
\operatorname{im}g_A=\ell(\operatorname{im}
(G:H_0(C(L_-))\to H_{-1}(Q)))\subset E.
\]

\noindent On total homology, the kernel of
\[
H:H_{-1}(Q)\longrightarrow H_{-1}(C(L_+))
\]
is the one-dimensional image of $H_0(C(L_-))$ under $G$, generated
by $G[\xi]$. Choose $r-1$ linear combinations of the $z_i$ whose
classes complement this kernel. Their images under $h$ are cycles
representing a basis of $H_{-1}(C(L_+))$.

The chosen leading terms complement $\operatorname{im}g_A$ in
$E$. Since $\ker h_A=\operatorname{im}g_A$ by
\eqref{eq:rank-one-sequence}, their images under $h_A$ are independent,
giving the required leading terms in grading $A$ and completing the induction.
\end{proof}

Setting $|L|=1$ gives Theorem~\ref{thm:knots}, and $p(K)=1$
characterizes prime knots. The Euler characteristic in Alexander
grading $g-1$ gives \eqref{eq:alexander-coefficient} and the ensuing
decomposition. For a prime knot, the next-to-top term has dimension
$a_{g-1}(K)+2$; \cite[Theorem~1.2]{Nifp2} gives the bound of
$a_{g-1}(K)+1$ fixed points in Corollary~\ref{cor:fixed-points}.

\begin{remark}\label{rem:sqp-concentration}
For a nonsplit strongly quasipositive link, there is a surviving
top class in Maslov grading zero \cite[Proposition~5.19]{HR22},
but the entire top term need not be concentrated there.
For example, the strongly quasipositive knot $K=12n_{293}$
 has genus two and satisfies
$\HFK(K,2)\cong\F[0]\oplus\F[1]$,  \cite{Knotinfo}.

\end{remark}

\bibliographystyle{alpha}
\bibliography{references}

@article{BD23,
  author = {Binns, Fraser and Dey, Subhankar},
  title = {{Rank bounds in link Floer homology and detection results}},
  journal = {Quantum Topol.},
  volume = {16},
  number = {4},
  pages = {743--794},
  year = {2025},
  doi = {10.4171/QT/197},
  eprint = {2201.03048},
  archiveprefix = {arXiv},
  primaryclass = {math.GT},
  url = {https://doi.org/10.4171/QT/197},
}

@inproceedings{Buskirk,
  author = {Van Buskirk, James M.},
  editor = {Rolfsen, Dale},
  title = {{Positive knots have positive Conway polynomials}},
  booktitle = {Knot Theory and Manifolds},
  pages = {146--159},
  publisher = {Springer Berlin Heidelberg},
  address = {Berlin, Heidelberg},
  year = {1985},
  isbn = {978-3-540-39616-1},
}

@misc{CH,
  author = {Cheng, Zhechi and Hedden, Matthew},
  title = {{Knot Floer homology of positive braids}},
  year = {2026},
  eprint = {2504.13005v2},
  archiveprefix = {arXiv},
  primaryclass = {math.GT},
  howpublished = {arXiv:2504.13005v2},
  url = {https://arxiv.org/abs/2504.13005v2},
}

@article{CHS22,
  author = {Cheng, Zhechi and Hedden, Matthew and Sarkar, Sucharit},
  title = {{Murasugi sum and extremal knot Floer homology}},
  journal = {Quantum Topol.},
  volume = {16},
  number = {4},
  pages = {795--830},
  year = {2025},
  doi = {10.4171/QT/186},
  eprint = {2202.09041},
  archiveprefix = {arXiv},
  primaryclass = {math.GT},
  url = {https://doi.org/10.4171/QT/186},
}

@article{FLL,
  author = {Feller, Peter and Lewark, Lukas and Lobb, Andrew},
  title = {{Almost positive links are strongly quasipositive}},
  journal = {Math. Ann.},
  volume = {385},
  pages = {481--510},
  year = {2023},
  doi = {10.1007/s00208-021-02328-x},
  eprint = {1809.06692},
  archiveprefix = {arXiv},
  primaryclass = {math.GT},
  url = {https://doi.org/10.1007/s00208-021-02328-x},
}

@incollection{Gabai,
  author = {Gabai, David},
  title = {{The Murasugi sum is a natural geometric operation}},
  booktitle = {Low-dimensional topology (San Francisco, Calif., 1981)},
  series = {Contemp. Math.},
  volume = {20},
  pages = {131--143},
  publisher = {Amer. Math. Soc., Providence, R.I.},
  year = {1983},
  doi = {10.1090/conm/020/718138},
  url = {https://doi.org/10.1090/conm/020/718138},
}

@article{HKM07,
  author = {Honda, Ko and Kazez, William H. and Mati\'{c}, Gordana},
  title = {{Right-veering diffeomorphisms of compact surfaces with boundary}},
  journal = {Invent. Math.},
  volume = {169},
  pages = {427--449},
  year = {2007},
  doi = {10.1007/s00222-007-0051-4},
  eprint = {math/0510639},
  archiveprefix = {arXiv},
  primaryclass = {math.GT},
  url = {https://doi.org/10.1007/s00222-007-0051-4},
}

@article{HR22,
  author = {Hedden, Matthew and Raoux, Katherine},
  title = {{Knot Floer homology and relative adjunction inequalities}},
  journal = {Sel. Math. New Ser.},
  volume = {29},
  number = {1},
  pages = {7},
  year = {2023},
  doi = {10.1007/s00029-022-00810-1},
  eprint = {2009.05462},
  archiveprefix = {arXiv},
  primaryclass = {math.GT},
  url = {https://doi.org/10.1007/s00029-022-00810-1},
}

@article{HW18,
  author = {Hedden, Matthew and Watson, Liam},
  title = {{On the geography and botany of knot Floer homology}},
  journal = {Sel. Math. New Ser.},
  volume = {24},
  pages = {997--1037},
  year = {2018},
}

@article{Ito22,
  author = {Ito, Tetsuya},
  title = {{A note on HOMFLY polynomial of positive braid links}},
  journal = {Int. J. Math.},
  volume = {33},
  number = {4},
  pages = {2250031},
  year = {2022},
  doi = {10.1142/S0129167X22500318},
  eprint = {2005.13188},
  archiveprefix = {arXiv},
  primaryclass = {math.GT},
  url = {https://doi.org/10.1142/S0129167X22500318},
}

@article{K00,
  author = {Khovanov, Mikhail},
  title = {{A categorification of the Jones polynomial}},
  journal = {Duke Math. J.},
  volume = {101},
  pages = {359--426},
  year = {2000},
  number = {3},
}

@misc{KMMS,
  author = {Kegel, Marc and Manikandan, Naageswaran and Mousseau, Leo and Silvero, Marithania},
  title = {{Khovanov homology of positive links and of L-space knots}},
  year = {2023},
  eprint = {2304.13613v2},
  archiveprefix = {arXiv},
  primaryclass = {math.GT},
  howpublished = {arXiv:2304.13613v2},
  url = {https://arxiv.org/abs/2304.13613v2},
}

@article{Khovanov03,
  author = {Khovanov, Mikhail},
  title = {{Patterns in knot cohomology I}},
  journal = {Experiment. Math.},
  volume = {12},
  number = {3},
  pages = {365--374},
  year = {2003},
  eprint = {math/0201306},
  archiveprefix = {arXiv},
  primaryclass = {math.QA},
  url = {https://arxiv.org/abs/math/0201306},
}

@misc{Knotinfo,
  author = {Livingston, Charles and Moore, Allison H.},
  title = {{KnotInfo: Table of Knot Invariants and LinkInfo: Table of Link Invariants}},
  howpublished = {\url{https://knotinfo.org/}},
  url = {https://knotinfo.org/},
  year = {2026},
  month = {November},
}

@article{Lee05,
  author = {Lee, Eun Soo},
  title = {{An endomorphism of the Khovanov invariant}},
  journal = {Adv. Math.},
  volume = {197},
  number = {2},
  pages = {554--586},
  year = {2005},
  doi = {10.1016/j.aim.2004.10.015},
  eprint = {math/0210213},
  archiveprefix = {arXiv},
  url = {https://doi.org/10.1016/j.aim.2004.10.015},
}

@misc{Manikandan26,
  author = {Manikandan, Naageswaran},
  title = {{Reduced even and odd Khovanov homology of positive and $0$-adequate links}},
  year = {2026},
  howpublished = {arXiv:2610.08583v1},
  eprint = {2610.08583v1},
  archiveprefix = {arXiv},
  primaryclass = {math.GT},
  url = {https://arxiv.org/abs/2610.08583v1},
}

@article{Nifp2,
  author = {Ni, Yi},
  title = {{Knot Floer homology and fixed points}},
  journal = {J. Topol.},
  volume = {19},
  number = {2},
  pages = {e70077},
  year = {2026},
  doi = {10.1112/topo.70077},
  url = {https://doi.org/10.1112/topo.70077},
}

@article{OScontact,
  author = {Ozsv{\'{a}}th, Peter and Szab{\'{o}}, Zolt{\'{a}}n},
  title = {{Heegaard Floer homology and contact structures}},
  journal = {Duke Math. J.},
  volume = {129},
  pages = {39--61},
  year = {2005},
}

@article{OSknot,
  author = {Ozsv{\'a}th, Peter and Szab{\'o}, Zolt{\'a}n},
  title = {{Holomorphic disks and knot invariants}},
  journal = {Adv. Math.},
  volume = {186},
  number = {1},
  pages = {58--116},
  year = {2004},
  doi = {10.1016/j.aim.2003.05.001},
  url = {https://doi.org/10.1016/j.aim.2003.05.001},
  issn = {0001-8708},
  mrnumber = {2065507},
  mrclass = {57M27 (57R58)},
}

@article{OSlink,
  author = {Ozsv{\'{a}}th, Peter and Szab{\'{o}}, Zolt{\'{a}}n},
  title = {{Holomorphic disks, link invariants and the multi-variable Alexander polynomial}},
  journal = {Algebr. Geom. Topol.},
  volume = {8},
  pages = {615--692},
  year = {2008},
}

@article{Ozawa,
  author = {Ozawa, Makoto},
  title = {{Closed incompressible surfaces in the complements of positive knots}},
  journal = {Comment. Math. Helv.},
  volume = {77},
  number = {2},
  pages = {235--243},
  year = {2002},
}

@article{Rasmussen2010,
  author = {Rasmussen, Jacob},
  title = {{Khovanov homology and the slice genus}},
  journal = {Invent. Math.},
  volume = {182},
  number = {2},
  pages = {419--447},
  year = {2010},
  doi = {10.1007/s00222-010-0275-6},
  eprint = {math/0402131},
  archiveprefix = {arXiv},
  primaryclass = {math},
  url = {https://doi.org/10.1007/s00222-010-0275-6},
  issn = {00209910},
}

@article{Rasmussen2015,
  author = {Rasmussen, Jacob},
  title = {{Some differentials on Khovanov--Rozansky homology}},
  journal = {Geom. Topol.},
  volume = {19},
  number = {6},
  pages = {3031--3104},
  year = {2015},
  doi = {10.2140/gt.2015.19.3031},
  eprint = {math/0607544},
  archiveprefix = {arXiv},
  url = {https://doi.org/10.2140/gt.2015.19.3031},
}

@phdthesis{Rthesis,
  author = {Rasmussen, Jacob},
  title = {{Floer homology and knot complements}},
  school = {Harvard University},
  year = {2003},
}

@article{Shumakovitch14,
  author = {Shumakovitch, Alexander N.},
  title = {{Torsion of Khovanov homology}},
  journal = {Fund. Math.},
  volume = {225},
  number = {1},
  pages = {343--364},
  year = {2014},
  doi = {10.4064/fm225-1-16},
  eprint = {math/0405474},
  archiveprefix = {arXiv},
  url = {https://doi.org/10.4064/fm225-1-16},
}

@article{Stoimenow03,
  author = {Stoimenow, Alexander},
  title = {{Positive knots, closed braids and the Jones polynomial}},
  journal = {Ann. Sc. Norm. Super. Pisa Cl. Sci. (5)},
  volume = {2},
  number = {2},
  pages = {237--285},
  year = {2003},
  url = {https://www.numdam.org/item/ASNSP_2003_5_2_2_237_0/},
}

@article{Stoimenow05,
  author = {Stoimenow, Alexander},
  title = {{On polynomials and surfaces of variously positive links}},
  journal = {J. Eur. Math. Soc.},
  volume = {7},
  number = {4},
  pages = {477--509},
  year = {2005},
  doi = {10.4171/JEMS/36},
  url = {https://doi.org/10.4171/JEMS/36},
}

@article{Wehrli04,
  author = {Wehrli, Stephan M.},
  title = {{A spanning tree model for Khovanov homology}},
  journal = {J. Knot Theory Ramifications},
  volume = {17},
  number = {12},
  pages = {1561--1574},
  year = {2008},
  eprint = {math/0409328v2},
  archiveprefix = {arXiv},
  url = {https://arxiv.org/abs/math/0409328v2},
}

@article{DasbachLin06,
  author = {Dasbach, Oliver T. and Lin, Xiao-Song},
  title = {{On the head and the tail of the colored Jones polynomial}},
  journal = {Compos. Math.},
  volume = {142},
  number = {5},
  pages = {1332--1342},
  year = {2006},
  doi = {10.1112/S0010437X06002296},
  eprint = {math/0604230},
  archiveprefix = {arXiv},
  primaryclass = {math.GT},
  url = {https://doi.org/10.1112/S0010437X06002296},
}

@misc{Buchanan22,
  author = {Buchanan, Lizzie},
  title = {{A new condition on the Jones polynomial of a fibered positive link}},
  year = {2022},
  howpublished = {arXiv:2204.03846v2},
  eprint = {2204.03846v2},
  archiveprefix = {arXiv},
  primaryclass = {math.GT},
  url = {https://arxiv.org/abs/2204.03846v2},
}

@article{BN02,
  author = {Bar-Natan, Dror},
  title = {{On Khovanov's categorification of the Jones polynomial}},
  journal = {Algebr. Geom. Topol.},
  volume = {2},
  pages = {337--370},
  year = {2002},
  doi = {10.2140/agt.2002.2.337},
  eprint = {math/0201043},
  archiveprefix = {arXiv},
  primaryclass = {math.QA},
  url = {https://doi.org/10.2140/agt.2002.2.337},
}

@article{PrzytyckiSazdanovic14,
  author = {Przytycki, J{\'o}zef H. and Sazdanovi{\'c}, Radmila},
  title = {{Torsion in Khovanov homology of semi-adequate links}},
  journal = {Fund. Math.},
  volume = {225},
  pages = {277--303},
  year = {2014},
  doi = {10.4064/fm225-1-13},
  eprint = {1210.5254},
  archiveprefix = {arXiv},
  url = {https://doi.org/10.4064/fm225-1-13},
}

@article{Kindred18,
  author = {Kindred, Thomas},
  title = {{Plumbing essential states in Khovanov homology}},
  journal = {New York J. Math.},
  volume = {24},
  pages = {588--610},
  year = {2018},
  url = {https://nyjm.albany.edu/j/2018/24-30.html},
}

@article{DunfieldLipshitzSchuetz26,
  author = {Dunfield, Nathan M. and Lipshitz, Robert and Sch{\"u}tz, Dirk},
  title = {Local equivalence and refinements of {Rasmussen}'s $s$-invariant},
  journal = {J. Topol.},
  volume = {19},
  number = {1},
  pages = {e70056},
  year = {2026},
  doi = {10.1112/topo.70056},
  eprint = {2312.09114},
  archiveprefix = {arXiv},
  url = {https://doi.org/10.1112/topo.70056}
}

@article{VMS25,
  author = {Del Valle V{\'i}lchez, {\'A}lvaro and Gonz{\'a}lez-Meneses, Juan and Silvero, Marithania},
  title = {{Positive 3-braids, Khovanov homology and Garside theory}},
  journal = {Rev. Mat. Complut.},
  year = {2026},
  doi = {10.1007/s13163-026-00585-5},
  eprint = {2504.06194},
  archiveprefix = {arXiv},
  primaryclass = {math.GT},
  url = {https://doi.org/10.1007/s13163-026-00585-5},
}

@article{ZemkeGradings19,
  author = {Zemke, Ian},
  title = {{Link cobordisms and absolute gradings on link Floer homology}},
  journal = {Quantum Topol.},
  volume = {10},
  number = {2},
  pages = {207--323},
  year = {2019},
  doi = {10.4171/QT/124},
  eprint = {1701.03454v3},
  archiveprefix = {arXiv},
  primaryclass = {math.GT},
  url = {https://doi.org/10.4171/QT/124},
}
\enlargethispage{2\baselineskip}
\end{document}